\documentclass[11pt]{article}
       \usepackage{amsfonts}
       \usepackage{stmaryrd}
       \usepackage{latexsym,amssymb,mathrsfs,fancyhdr}
       \font\tenmsb=msbm10
       \font\sevenmsb=msbm7
       \font\fivemsb=msbm5
       \catcode`\@=11
       \ifx\amstexloaded@\relax\catcode`\@=\active
       \endinput\else\let\amstexloaded@\relax\fi
       \def\spaces@{\space\space\space\space\space}
       \def\spaces@@{\spaces@\spaces@\spaces@\spaces@\spaces@}
       \def\space@.  {\futurelet\space@\relax}
       \space@.   %
       \def\Err@#1{\errhelp\defaulthelp@\errmessage{AmS-TeX error: #1}}
       \def\relaxnext@{\let\next\relax}
       \def\accentfam@{7}
       \def\noaccents@{\def\accentfam@{0}}
       \def\Cal{\relaxnext@\ifmmode\let\next\Cal@\else
       \def\next{\Err@{Use \string\Cal\space only in math mode}}\fi\next}
       \def\Cal@#1{{\Cal@@{#1}}}
       \def\Cal@@#1{\noaccents@\fam\tw@#1}
       \def\Bbb{\relaxnext@\ifmmode\let\next\Bbb@\else
       \def\next{\Err@{Use \string\Bbb\space only in math mode}}\fi\next}
       \def\Bbb@#1{{\Bbb@@{#1}}}
       \def\Bbb@@#1{\noaccents@\fam\msbfam#1}
       \newfam\msbfam
       \textfont\msbfam=\tenmsb
       \scriptfont\msbfam=\sevenmsb
       \scriptscriptfont\msbfam=\fivemsb
\usepackage{dsfont}
\usepackage[german,english]{babel}
\usepackage{amsmath,amssymb}
\usepackage[square, comma, sort&compress, numbers]{natbib}
\usepackage{cases}
\usepackage{mathrsfs}
\usepackage{amsmath}
\usepackage{amsthm}
\usepackage{amsfonts}
\usepackage{amssymb}
\usepackage{latexsym}
\usepackage{fancyhdr}
\usepackage{extarrows}
\usepackage{geometry}
\usepackage{color}
\usepackage[comma,numbers,square,sort&compress]{natbib}
\usepackage{epstopdf}
\usepackage{graphicx,amsmath}
\usepackage{subfigure}
\usepackage{amssymb}
\theoremstyle{plain}
\newtheorem{theorem}{Theorem}[section]
\newtheorem{corollary}[theorem]{Corollary}
\newtheorem{lemma}[theorem]{Lemma}
\newtheorem{example}[theorem]{Example}

\numberwithin{equation}{section}
\newtheorem{remark}[theorem]{Remark}

\theoremstyle{definition}
\newtheorem{definition}[theorem]{Definition}

\large\normalsize

\usepackage{enumerate}

\begin{document}
\numberwithin{equation}{section}
\title{{\bf Further characterizations and  representations of the Minkowski inverse in Minkowski space}}


\author{{\ Jiale Gao $^{a}$, \ Qingwen Wang $^{b}$, \ Kezheng Zuo $^{a}$\footnote{E-mail address: xiangzuo28@163.com (K. Zuo).}, \ Jiabao Wu $^{a}$}
\\{{\small $^a$ School of Mathematics and Statistics, Hubei Normal University, Huangshi, 435002, PR China}}
\\ {\small $^b$ College of Science, Shanghai University, Shanghai, 200444, PR China}}

\maketitle

\begin{center}
\begin{minipage}{135mm}
{{\small {\bf Abstract:}  This paper is aimed to  identify some new characterizations and  representations of the  Minkowski inverse in Minkowski space. First of all, a few representations of $\{1,3^{\mathfrak{m}}\}$,
$\{1,2,3^{\mathfrak{m}}\}$, $\{1,4^{\mathfrak{m}}\}$ and $\{1,2,4^{\mathfrak{m}}\}$-inverses are given in order to represent the Minkowski inverse. Secondly, some famous characterizations of the Moore-Penrose inverse are extended to that of the Minkowski inverse. Thirdly, using the Hartwig-Spindelb\"{o}ck decomposition we present  a representation of the Minkowski inverse. And, based on this result,   an interesting characterization of the Minkowski inverse is showed   by a rank equation. Finally, we obtain several new representations of the Minkowski inverse in a more general form, by which the Minkowski inverse of a class of block matrices is given.
\begin{description}
\item[{\bf Key words}:]  Minkowski inverse; Minkowski space; Hartwig-Spindelb\"{o}ck decomposition; Full rank factorization
\item[{\bf AMS subject classifications}:]15A03; 15A09; 15A24 
\end{description}
}}
\end{minipage}
\end{center}

\section{Introduction}\label{introductionsection}

In order to easily test that a Mueller matrix maps the forward light cone into itself in studying polarized hight, Renardy   \cite{minSVDre} explored the singular value decomposition in Minkowski space. Subsequently, Meenakshi \cite{minMPre}
defined the Minkowski inverse in Minkowski space,  and  gave a condition for a Mueller matrix to have a singular value decomposition in terms of its Minkowski inverse. Since this article came out, the generalized inverses in Minkowski space have attracted considerable attention. Zekraoui et al. \cite{minMPfurre} derived   some new  algebraic and topological properties of the Minkowski inverse. Meenakshi \cite{minEPre} introduced the concept of a range symmetric matrix  in Minkowski space, which was further studied by \cite{minkEPre,minMPEPre,minEPblockre}. The Minkowski inverse has been  widely used in many applications such as the anti-reflexive  solutions of matrix equations \cite{appequre} and  matrix partial orderings \cite{partordmEPre,partordre}. The weighted Minkowski inverse defined by \cite{WminMPorire} is a generalization of the Minkowski inverse, and many of its properties, representations and approximations were established in \cite{WminMPfurre,WminMPorire,WminMPiterre}. Recently, Wang et al.  introduced the $\mathfrak{m}$-core inverse \cite{mincorere}, the $\mathfrak{m}$-core-EP inverse \cite{mincoreEPre} and the $\mathfrak{m}$-WG inverse \cite{minwgroupre} in Minkowski space, which are viewed as  generalizations of the core inverse,   the core-EP inverse  and   the weak group inverse, respectively.
\par
It is well known that the Moore-Penrose inverse \cite{penrosere} not only  plays an irreplaceable role in solving linear matrix equations, but also  is
 a generally accepted tool in statistics,  extreme-value  problems and other scientific disciplines. Moreover, this inverse pervades  a great number  of mathematical fields: $C^*$-algebras, rings, Hilbert spaces, Banach spaces,  categories, tensors, and the quaternion skew field. The algebraic properties, characterizations, representations, perturbation theory, and iterative computations of the   Moore-Penrose inverse have been  extensively investigated.
For more details on the study of the Moore-Penrose inverse refer to books \cite{ringbookre,benbookre,cambookre,drabookre,Ban02re,wangbookre}.
\par
Although the Minkowski inverse in Minkowski space can be regarded as an extension of the Moore-Penrose inverse, there are many differences between these two classes of generalized inverses, especially in their existence conditions (see \cite{WminMPfurre,minMPre,minMPfurre}). So, it's natural to ask what interesting results  for Minkowski inverse can be drawn by considering  some known conclusions of the Moore-Penrose inverse.
\par
Mainly inspired by \cite{solution,ZlobecMPref,Centralizerref,syl},  we summarize the main topics of this work as below:
\begin{itemize}
  \item A few characterizations and representations of $\{1,3^{\mathfrak{m}}\}$,
$\{1,2,3^{\mathfrak{m}}\}$, $\{1,4^{\mathfrak{m}}\}$ and $\{1,2,4^{\mathfrak{m}}\}$-inverses are showed.
  \item We apply the solvability of matrix equation, the nonsingularity of matrices, the existence of projectors, and the index of matrices to characterize the existence of the Minkowski inverse, which extends  some classic characterizations of the Moore-Penrose inverse in $\mathbb{C}^{m\times n}$ with usual Hermitian adjoint and in a ring with   involution. And, we show various representations of the  Minkowski inverse in different cases.

  \item Using the Hartwig-Spindelb\"{o}ck decomposition, we present a new representation of the Minkowski inverse. Based on this result, an interesting characterization of the Minkowski inverse is presented by a rank equation.

\item   Motivated by the Zlobec formula of the Moore-Penrose inverse, we give a more general representation  of the Minkowski inverse, and   apply  it to compute the Minkowski inverse of a class of block matrices.
\end{itemize}
\par
This paper is organized as follows. Section \ref{Nottermsec} presents notations and terminology. In  Section \ref{Preliminariessec},   some necessary lemmas are given. We  devote Section \ref{character1234sec} to the  characterizations of  $\{1,3^{\mathfrak{m}}\}$,
$\{1,2,3^{\mathfrak{m}}\}$, $\{1,4^{\mathfrak{m}}\}$ and $\{1,2,4^{\mathfrak{m}}\}$-inverses.
Some classic  properties of the Moore-Penrose inverse are extended to the case of the Minkowski inverse in Section \ref{charactermimMPsec}.
 In Section \ref{characterminMPfursec}, we further extend several   characterizations of the Moore-Penrose inverse in a ring  to the Minkowski inverse.
 We characterize the Minkowski   inverse by a rank equation in Section \ref{characterrankminMPsec}. Section \ref{generalizingZminMPsec} focuses  on showing a few  new representations of the Minkowski inverse.

\section{Notations and terminology}\label{Nottermsec}
Throughout this paper, we adopt the following notations and terminology. Let $\mathbb{C}^{n}$, $\mathbb{C}^{m\times n}$, and  $\mathbb{C}^{m\times n}_r$ be the sets of all complex $n$-dimensional vectors,  complex  ${m\times n}$ matrices, and complex  ${m\times n}$ matrices with rank $r$, respectively.
The symbols $A^*$, $\mathcal{R}(A)$, $\mathcal{N}(A)$, ${\rm rank}( A )$,  $A^{-1}_{R}$, and  $A^{-1}_{L}$ stand for the conjugate transpose, the range, the null space, the rank,  a right inverse, and a left inverse of $A\in\mathbb{C}^{ m \times n}$, respectively.
The index of $A\in\mathbb{C}^{n\times n}$, denoted by ${\rm Ind}(A)$, is the smallest nonnegative integer $t$ satisfying ${\rm rank}(A^{t+1})={\rm rank}(A^{t})$.
And, $A^0=I_n$ for $A\in\mathbb{C}^{n\times n}$, where $I_n$ is the identity matrix in $\mathbb{C}^{n\times n}$. We denote the dimension and the orthogonal complementary subspace of  a subspace $\mathcal{L}\subseteq \mathbb{C}^{n}$ by ${\rm dim}(\mathcal{L})$  and $\mathcal{L}^{\perp}$, respectively.
By $P_{\mathcal{S},\mathcal{T}}$ we denote the projector  onto $\mathcal{S}$ along $\mathcal{T}$, where two subspaces $\mathcal{S},\mathcal{T}\subseteq\mathbb{C}^{n}$ satisfy that the direct sum of $\mathcal{S}$ and $\mathcal{T}$ is $\mathbb{C}^{n}$, i.e., $\mathcal{S}\oplus \mathcal{T}=\mathbb{C}^{n}$. In particular, $P_{\mathcal{S}}=P_{\mathcal{S},\mathcal{S}^{\perp}}$.
\par
The Moore-Penrose inverse \cite{penrosere}   of $A\in\mathbb{C}^{m\times n}$ is the unique matrix $X\in\mathbb{C}^{n\times m}$ verifying
\begin{equation*}
   AXA=A,  XAX=X,  (AX)^*=AX, (XA)^*=XA,
\end{equation*}
and is denoted by $A^{\dag}$. The group inverse \cite{grouporgre} of $A\in\mathbb{C}^{n\times n}$ is the unique matrix $X\in\mathbb{C}^{n\times n}$ satisfying
\begin{equation*}
  AXA=X, XAX=X, AX=XA,
\end{equation*}
and is denoted by $A^{\#}$.
For $A\in\mathbb{C}^{m\times n}$, if there is a matrix $X\in\mathbb{C}^{n \times m}$ satisfying
\begin{equation*}
XAX=X, \mathcal{R}(X)=\mathcal{T},\mathcal{N}(X)=\mathcal{S},
\end{equation*}
where $\mathcal{T}\subseteq\mathbb{C}^n$ and $\mathcal{S}\subseteq\mathbb{C}^m$ are two subspaces, then $X$ is unique and is denoted  by $A_{\mathcal{T},\mathcal{S}}^{(2)}$ \cite{benbookre,wangbookre}. Particularly, if $AA_{\mathcal{T},\mathcal{S}}^{(2)}A=A$,  we denote $A_{\mathcal{T},\mathcal{S}}^{(1,2)}=A_{\mathcal{T},\mathcal{S}}^{(2)}$.
\par
Additionally, the Minkowski inner product  \cite{minMPre,minSVDre} of
two elements $x$ and $y$ in $\mathbb{C}^{n}$
is defined by $(x,y)=<x,Gy>$, where
$
G=\left(
    \begin{array}{cc}
      1 & 0 \\
      0 & -I_{n-1} \\
    \end{array}
  \right)
$ represents the Minkowski metric matrix, and  $<\cdot,\cdot>$ is the conventional Euclidean inner product. The complex linear space $\mathbb{C}^{n}$  with Minkowski inner product is called the Minkowski space. Notice  that the Minkowski space is also an  indefinite inner product space \cite{indef,indef2324re}. The Minkowski adjoint of $A\in\mathbb{C}^{m\times n}$ is $A^{\sim}=GA^*F$, where $G$ and $F$ are Minkowski metric matrices of orders $n$ and $m$, respectively.

\begin{definition}\cite{minMPre,indef2324re}\label{minMPdefintion}
Let $A\in\mathbb{C}^{m\times n}$.
\begin{enumerate}[$(1)$]
       \item If there exists $X\in\mathbb{C}^{n\times m}$ such that the following equations
\begin{equation*}
  (1)AXA=A, (2)XAX=X, (3^{\mathfrak{m}})(AX)^{\sim}=AX, (4^{\mathfrak{m}})(XA)^{\sim}=XA,
\end{equation*}
then $X$ is called the Minkowski inverse of $A$, and is denoted by $A^{\mathfrak{m}}$.
\item If $X\in\mathbb{C}^{n\times m}$ satisfies equations $(i),(j),...,(k)$ from among  equations $(1)$--$(4^{\mathfrak{m}})$, then $X$ is called a $\{i,j,...,k\}$-inverse  of $A$, and is denoted by $A^{(i,j,...,k)}$. The set of all $\{i,j,...,k\}$-inverses of $A$ is denoted by $A\{i,j,...,k\}$.
\end{enumerate}
\end{definition}

\section{Preliminaries}\label{Preliminariessec}
This section  begins with  recalling existence conditions and some basic properties of the Minkowski inverse, which will be useful in the later discussion.

\begin{lemma}\cite[Theorem 1]{minMPre}\label{minMPexistconlem}
Let $A\in\mathbb{C}^{m\times n}$. Then $A^{\mathfrak{m}}$ exists if and only if ${\rm rank}(AA^{\sim})={\rm rank}(A^{\sim}A)={\rm rank}(A)$.
\end{lemma}

\begin{lemma}\cite[Theorem 8]{minMPfurre}\label{minMPfullreprelemma}
Let $A\in\mathbb{C}^{m\times n}_r$ and $r > 0$, and let $A=BC$ be a full rank factorization of $A$,
where $B\in\mathbb{C}^{m\times r}_r$ and $C\in\mathbb{C}^{r\times n}_r$.
If $A^{\mathfrak{m}}$ exists,  then $A^{\mathfrak{m}}=C^{\sim}(CC^{\sim})^{-1}(B^{\sim}B)^{-1}B^{\sim}$.
\end{lemma}

\begin{lemma}\label{mMPprole}\cite[Theorem 9]{indef}
Let $A\in\mathbb{C}^{m\times n}$ with ${\rm rank}(AA^{\sim})={\rm rank}(A^{\sim}A)={\rm rank}(A)$. Then
\begin{enumerate}[$(1)$]
\item\label{mMPRN} $\mathcal{R}(A^{\mathfrak{m}})=\mathcal{R}(A^{\sim})$ and $\mathcal{N}(A^{\mathfrak{m}})=\mathcal{N}(A^{\sim})$;
  \item\label{AAmMPeq} $AA^{\mathfrak{m}}=P_{\mathcal{R}(A),\mathcal{N}(A^{\sim})}$;
  \item\label{AmAMPeq} $A^{\mathfrak{m}}A=P_{\mathcal{R}(A^{\sim}),\mathcal{N}(A)}$.
\end{enumerate}
\end{lemma}

\begin{remark}\label{minMPATS2rem}
Under the hypotheses of Lemma \ref{mMPprole}, we immediately  have
 \begin{equation}\label{minMPATS2}
A^{\mathfrak{m}}=A^{(1,2)}_{\mathcal{R}(A^{\sim}),\mathcal{N}(A^{\sim})}.
\end{equation}
\end{remark}
\par
Furthermore, we recall an important application of $\{1\}$-inverses to solve matrix equations.

\begin{lemma}\cite[Theorem 1.2.5]{wangbookre}\label{AXDsolle}
Let  $A\in\mathbb{C}^{m\times n}$, $B\in\mathbb{C}^{p\times q}$ and $D\in\mathbb{C}^{m\times q}$. Then there is a solution $X\in\mathbb{C}^{n\times p}$ to the matrix equation $AXB=D$ if and only if,  for some $A^{(1)}\in A\{1\}$ and $B^{(1)}\in B\{1\}$,  $AA^{(1)}DB^{(1)}B=D$, in which case, the general solution is
\begin{equation*}
  X=A^{(1)}DB^{(1)}+(I_n-A^{(1)}A)Y+Z(I_p-BB^{(1)}),
\end{equation*}
 where $A^{(1)}\in A\{1\}$ and $B^{(1)}\in B\{1\}$ are fixed but arbitrary, and $Y\in\mathbb{C}^{n\times p}$ and $Z\in\mathbb{C}^{n\times p}$ are  arbitrary.
\end{lemma}
\par
Two significant results of $A_{\mathcal{T},\mathcal{S}}^{(2)}$ are reviewed in order to show existence conditions of the Minkowski inverse in Section \ref{charactermimMPsec} and to represent the Minkowski inverse in Section \ref{generalizingZminMPsec}, respectively.

\begin{lemma}\cite[Theorem 2.1]{rankatsre}\label{ATS2Gchale}
Let $A\in\mathbb{C}^{m\times n}_r$, and let
two subspaces $\mathcal{T}\subseteq \mathbb{C}^{n}$  and $\mathcal{S}\subseteq \mathbb{C}^{m}$  be such that ${\rm dim}(\mathcal{T})\leq r$ and ${\rm dim}(\mathcal{S})= m-{\rm dim}(\mathcal{T})$.
 Suppose that $H\in\mathbb{C}^{n\times m}$ is such that $\mathcal{R}(H)=\mathcal{T}$  and $\mathcal{N}(H)=\mathcal{S}$. If $A_{\mathcal{T},\mathcal{S}}^{(2)}$ exists, then ${\rm Ind}(AH)={\rm Ind}(HA)=1$. Further, we have $A_{\mathcal{T},\mathcal{S}}^{(2)}=(HA)^{\#}H=H(AH)^{\#}$.
\end{lemma}

\begin{lemma}[Urquhart formula, \cite{Urquhartref}]\label{Urquhartth}
Let $A\in\mathbb{C}^{m\times n}$, $U\in\mathbb{C}^{n\times p}$, $V\in\mathbb{C}^{q\times m}$, and
\begin{equation*}
X=U(VAU)^{(1)}V,
\end{equation*}
where $(VAU)^{(1)}\in(VAU)\{1\}$. Then
$X=A^{(1,2)}_{\mathcal{R}(U),\mathcal{N}(V)}$ if and only if ${\rm rank}(VAU)={\rm rank}(U)={\rm rank}(V)={\rm rank}(A)$.
\end{lemma}
\par
The following three auxiliary lemmas are critical to conclude results of Section \ref{characterrankminMPsec}.

\begin{lemma}[Hartwig-Spindelb\"{o}ck decomposition, \cite{HSdec}]\label{HSth} Let $A\in\mathbb{C}^{n\times n}_r$. Then $A$ can be represented in the form
\begin{equation}\label{HSAdec}
A=U\left(
     \begin{array}{cc}
       \Sigma K & \Sigma L \\
       0 & 0 \\
     \end{array}
   \right)U^*,
\end{equation}
where $U\in\mathbb{C}^{n\times n}$ is unitary,
$\Sigma ={\rm diag}(\sigma_1,\sigma_2,...,\sigma_r)$ is the diagonal matrix of singular values of $A$, $\sigma_i>0(i=1,2,...,r)$, and $K\in\mathbb{C}^{r\times r}$ and $L\in\mathbb{C}^{r\times (n-r)}$ satisfy
\begin{equation}\label{HSKLconditioneq}
KK^*+LL^*=I_r.
\end{equation}
\end{lemma}

\begin{lemma}\cite[Theorem 1]{solution}\label{raneconle}
Let $A\in\mathbb{C}^{m\times n}$, $B\in\mathbb{C}^{m\times m}$ and $C\in\mathbb{C}^{n\times n}$. Then there exists a solution $X\in\mathbb{C}^{n\times m}$ to the rank equation \begin{equation}\label{rankeq}
{\rm rank}\left(
            \begin{array}{cc}
              A & B \\
              C & X \\
            \end{array}
          \right)={\rm rank}(A)
\end{equation}
 if and only if $\mathcal{R}(B)\subseteq\mathcal{R}(A)$
 and $\mathcal{R}(C^*)\subseteq\mathcal{R}(A^*)$, in which case,
 \begin{equation}\label{ranksolution}
   X=CA^{\dag}B.
 \end{equation}
\end{lemma}

\begin{lemma}\cite[Theorem 1]{Liao}\label{XAYBsolle}
Let $A\in\mathbb{C}^{m\times n}_r$  and $B\in\mathbb{C}^{l\times h}_{r_1}$. Assume
$A=P\left(
      \begin{array}{cc}
        I_r & 0 \\
       0 & 0 \\
      \end{array}
    \right)
Q$ and
$B=P_1\left(
      \begin{array}{cc}
        I_{r_1} & 0 \\
       0 & 0 \\
      \end{array}
    \right)
Q_1$,
where  $P\in\mathbb{C}^{m\times m}$, $Q\in\mathbb{C}^{n\times n}$, $P_1\in\mathbb{C}^{l\times l}$ and $Q_1\in\mathbb{C}^{h\times h}$ are nonsingular.
If $r=r_1$, then the general solution of  the matrix equation
$
XAY=B
$
is given by
\begin{equation*}
X=P_1\left(
      \begin{array}{cc}
        X_1 & X_2 \\
       0 & X_4 \\
      \end{array}
    \right)P^{-1},
Y=Q^{-1}\left(
      \begin{array}{cc}
            X_1^{-1} & 0 \\
      Y_3 & Y_4 \\
      \end{array}
    \right)Q_1,
\end{equation*}
where  $X_{1}\in\mathbb{C}^{r\times r}$ is an arbitrary nonsingular matrix, and
$X_{2}\in\mathbb{C}^{r\times (m-r)}$,  $X_{4}\in\mathbb{C}^{(l-r)\times (m-l)}$, $Y_{3}\in\mathbb{C}^{(n-r)\times r}$ and $Y_{4}\in\mathbb{C}^{(n-r)\times (n-r)}$ are arbitrary.
\end{lemma}

\section{Characterizations of  $\{1,3^{\mathfrak{m}}\}$,
$\{1,2,3^{\mathfrak{m}}\}$, $\{1,4^{\mathfrak{m}}\}$ and $\{1,2,4^{\mathfrak{m}}\}$-inverses} \label{character1234sec}
An interesting conclusion proved by Kamaraj and Sivakumar   in \cite[Theorem 4]{indef} is that if $A\in\mathbb{C}^{m\times n}$ is such that $A^{\mathfrak{m}}$ exists,  then
\begin{equation}\label{minMPA13mA14meq}
A^{\mathfrak{m}}=A^{(1,4^{\mathfrak{m}})}AA^{(1,3^{\mathfrak{m}})},
\end{equation}
where $A^{(1,3^{\mathfrak{m}})}\in A\{1,3^{\mathfrak{m}}\}$ and  $A^{(1,4^{\mathfrak{m}})}\in A\{1,4^{\mathfrak{m}}\}$. This result shows the importance of $A^{(1,3^{\mathfrak{m}})}$ and $A^{(1,4^{\mathfrak{m}})}$  to represent the Minkowski inverse $A^{\mathfrak{m}}$. Moreover, Petrovi\'{c} and Stanimirovi\'{c} \cite{indef2324re} have investigated the representations and computations of $\{2,3^{\sim}\}$ and $\{{2,4^{\sim}}\}$-inverses in an indefinite inner product space, which are generalizations of $\{2,3^{\mathfrak{m}}\}$ and $\{{2,4^\mathfrak{m}}\}$-inverses in Minkowski space. Motivated by the above work,  we consider the characterizations of  $\{1,3^{\mathfrak{m}}\}$,
$\{1,2,3^{\mathfrak{m}}\}$, $\{1,4^{\mathfrak{m}}\}$ and $\{1,2,4^{\mathfrak{m}}\}$-inverses in this section.
Before staring, an auxiliary lemma is given as follows.

\begin{lemma}\label{mprojectormgongeth}
Let $A\in\mathbb{C}^{n\times s}$ and $B\in\mathbb{C}^{t\times n}$.
Then
\begin{equation*}
\left(P_{\mathcal{R}({A}),\mathcal{N}(B)}\right)^{\sim}=P_{\mathcal{R}(B^{\sim}),\mathcal{N}(A^{\sim})}.
\end{equation*}
\end{lemma}

\begin{proof}
Write $Q=(P_{\mathcal{R}({A}),\mathcal{N}(B)})^{\sim}=GP_{\left(\mathcal{N}(B)\right)^{\perp},\left(\mathcal{R}({A})\right)^{\perp}}G$. Then $Q^2=Q$,
\begin{align*}
  &\mathcal{R}(Q)= G\left(\mathcal{N}(B)\right)^{\perp}=\mathcal{R}(GB^*)=\mathcal{R}(B^{\sim}),\\
  &\left(\mathcal{N}(Q)\right)^{\perp}=\left(\mathcal{N}(P_{\left(\mathcal{N}(B)\right)^{\perp},\left(\mathcal{R}({A})\right)^{\perp}}G)\right)^{\perp} =\mathcal{R}(GP_{\mathcal{R}({A}),\mathcal{N}(B)})=\mathcal{R}(GA)=\mathcal{R}((A^{\sim})^{*}),
\end{align*}
which implies $\mathcal{N}(Q)=\left(\mathcal{R}((A^{\sim})^{*})\right)^{\perp}=\mathcal{N}(A^{\sim})$. Thus, $Q=P_{\mathcal{R}(B^{\sim}),\mathcal{N}(A^{\sim})}$.
\end{proof}

In the following theorems, we prove the equivalence of the existence of $\{1,3^{\mathfrak{m}}\}$ and $\{1,2,3^{\mathfrak{m}}\}$-inverses, and show some of their  characterizations.

\begin{theorem}\label{m13m123eqth}
Let $A\in\mathbb{C}^{m\times n}$. Then
 there exists  $X\in A\{1,3^{\mathfrak{m}}\}$ if and only if there exists $Y\in A\{1,2,3^{\mathfrak{m}}\}$.
\end{theorem}
\begin{proof}
``$\Leftarrow$". It is obvious. ``$\Rightarrow$". If there exists $X\in A\{1,3^{\mathfrak{m}}\}$, then
\begin{equation*}
  A=AXA=(AX)^{\sim}A=X^{\sim}A^{\sim}A,
\end{equation*}
which implies ${\rm rank}(A)={\rm rank}(A^{\sim}A)$. Using \cite[Theorem 2]{minMPre}, we have that there exists  $Y\in A\{1,2,3^{\mathfrak{m}}\}$.
\end{proof}

\begin{theorem}\label{m13ineqconditionth}
Let $A\in\mathbb{C}^{m\times n}$ and $X\in\mathbb{C}^{n\times m}$. Then the following statements are equivalent:
\begin{enumerate}[$(1)$]
  \item\label{m13ineqconditionit1} $X\in A\{1,3^{\mathfrak{m}}\}$;
  \item\label{m13ineqconditionit2} $A^{\sim}AX=A^{\sim}$;
  \item\label{m13ineqconditionit3} $AX=P_{\mathcal{R}(A),\mathcal{N}(A^{\sim})}$.
\end{enumerate}
In this case,
\begin{equation}\label{m13alleq}
  A\{1,3^{\mathfrak{m}}\}=\left\{
  \begin{array}{c|c}
  A^{(1,3^{\mathfrak{m}})}+(I_n-A^{(1,3^{\mathfrak{m}})}A)Y & Y\in\mathbb{C}^{n\times m} \\
  \end{array}
\right\},
\end{equation}
 where $A^{(1,3^{\mathfrak{m}})}\in A\{1,3^{\mathfrak{m}}\}$  is  fixed  but arbitrary.
\end{theorem}

\begin{proof}
\eqref{m13ineqconditionit1} $\Rightarrow$ \eqref{m13ineqconditionit2}. Since $X\in A\{1,3^{\mathfrak{m}}\}$,  it follows that $A^{\sim}A X=A^{\sim}(A X)^{\sim}=(AXA)^{\sim}=A^{\sim}$. \par
\eqref{m13ineqconditionit2} $\Rightarrow$ \eqref{m13ineqconditionit3}. Since $(AX)^{\sim}A=A$ from $A^{\sim}A X=A^{\sim}$, we have $AX=(AX)^{\sim}AX$, implying $(AX)^{\sim}=AX$. Thus,  $AX=(AX)^{\sim}AX=(AX)^2$, that is, $AX$ is a projector. Again by $(AX)^{\sim}=AX$, we have
 $AXA=A$, which, together  with $A^{\sim}A X=A^{\sim}$, shows that $\mathcal{R}(AX)=\mathcal{R}(A)$ and $\mathcal{N}(AX)=\mathcal{N}(A^{\sim})$. Hence, $AX=P_{\mathcal{R}(A),\mathcal{N}(A^{\sim})}$.\par
\eqref{m13ineqconditionit3} $\Rightarrow$ \eqref{m13ineqconditionit1}. Clearly, $AXA=P_{\mathcal{R}(A),\mathcal{N}(A^{\sim})}A=A$.
 Applying Lemma \ref{mprojectormgongeth} to $AX=P_{\mathcal{R}(A),\mathcal{N}(A^{\sim})}$, we see
 $(AX)^{\sim}=P_{\mathcal{R}(A),\mathcal{N}(A^{\sim})}=AX$.\par
In this case,  we have
$A\{1,3^{\mathfrak{m}}\}=\left\{
   \begin{array}{c|c}
   Z\in\mathbb{C}^{n\times m} & AZ=AA^{(1,3^{\mathfrak{m}})} \\
   \end{array}
 \right\}
$,
where $A^{(1,3^{\mathfrak{m}})}$ is a fixed  but arbitrary $\{1,3^{\mathfrak{m}}\}$-inverse of $A$. Thus, applying Lemma \ref{AXDsolle} to $AZ=AA^{(1,3^{\mathfrak{m}})}$, we have  \eqref{m13alleq} directly.
\end{proof}
\par
\begin{theorem}
Let $A\in\mathbb{C}^{m\times n}_r$ with ${\rm rank}(A^{\sim}A)={\rm rank}(A)>0$, and let   a full rank factorization of $A$ be $A =BC$, where $B\in\mathbb{C}^{m\times r}_r$ and $C\in\mathbb{C}^{r\times n}_r$. Then
\begin{equation*}
A\{1,2,3^{\mathfrak{m}}\}=\left\{
   \begin{array}{c|c}
     C_R^{-1}(B^{\sim}B)^{-1}B^{\sim} & C_R^{-1}\in\mathbb{C}^{n\times r}_{r} \\
   \end{array}
 \right\}.
\end{equation*}
\end{theorem}
\par
\begin{proof}
We can easily verify that $[C_R^{-1}(B^{\sim}B)^{-1}B^{\sim}]\in A\{1,2,3^{\mathfrak{m}}\}$. Conversely,
let $H\in A\{1,2,3^{\mathfrak{m}}\}$. Using the fact that
\begin{equation*}
  A\{1,2\}=\left\{
   \begin{array}{c|c}
     C_R^{-1}B^{-1}_{L}&  B^{-1}_{L}\in\mathbb{C}^{r\times m}_{r},C_R^{-1}\in\mathbb{C}^{n\times r}_{r}  \\
   \end{array}
 \right\},
\end{equation*}
we have $H=C^{-1}_{R}B^{-1}_{L}$ for some $B^{-1}_{L}\in\mathbb{C}^{r\times m}_{r}$ and $C_R^{-1} \in\mathbb{C}^{n\times r}_{r}$. Moreover, it follows from $H\in A\{3^{\mathfrak{m}}\}$ that
\begin{align*}
  (AH)^{\sim}=AH & \Leftrightarrow (BB^{-1}_{L})^{\sim}=BB^{-1}_{L} \\
   & \Leftrightarrow B^{-1}_{L}=(B^{\sim}B)^{-1}B^{\sim}.
\end{align*}
Hence every $H\in A\{1,2,3^{\mathfrak{m}}\}$ must be of the form $C_R^{-1}(B^{\sim}B)^{-1}B^{\sim}$. This completes the proof.
\end{proof}
\par
Using  an obvious fact that $X^{\sim}\in A^{\sim}\{1,3^{\mathfrak{m}}\} $ if and only if $ X\in A\{1,4^{\mathfrak{m}}\}$, where $A\in\mathbb{C}^{m\times n}$, we have  the following  results  which  show that  $\{1,4^{\mathfrak{m}}\}$ and $\{1,2,4^{\mathfrak{m}}\}$-inverses have  properties similar to that of
$\{1,3^{\mathfrak{m}}\}$ and $\{1,2,3^{\mathfrak{m}}\}$-inverses.
\par
\begin{theorem}
Let $A\in\mathbb{C}^{m\times n}$. Then
 there exists  $X\in A\{1,4^{\mathfrak{m}}\}$ if and only if there exists $Y\in A\{1,2,4^{\mathfrak{m}}\}$.
\end{theorem}
\par
\begin{theorem}\label{m14ineqconditionth}
Let $A\in\mathbb{C}^{m\times n}$ and $X\in\mathbb{C}^{n\times m}$. Then the following statements are equivalent:
\begin{enumerate}[$(1)$]
  \item $X\in A\{1,4^{\mathfrak{m}}\}$;
  \item $XAA^{\sim}=A^{\sim}$;
  \item $XA=P_{\mathcal{R}(A^{\sim}),\mathcal{N}(A)}$.
\end{enumerate}
In this case,
\begin{equation*}
A\{1,4^{\mathfrak{m}}\}=\left\{
   \begin{array}{c|c}
     A^{(1,4^{\mathfrak{m}})}+Z(I_m-AA^{(1,4^{\mathfrak{m}})}) &Z\in\mathbb{C}^{n\times m} \\
   \end{array}
 \right\},
\end{equation*}
 where $A^{(1,4^{\mathfrak{m}})}\in A\{1,4^{\mathfrak{m}}\}$  is  fixed  but arbitrary.
\end{theorem}
\par
\begin{theorem}
Let $A\in\mathbb{C}^{m\times n}_r$ with ${\rm rank}(AA^{\sim})=r>0$, and let   a full rank factorization of $A$ be $A =BC$, where $B\in\mathbb{C}^{m\times r}_r$ and $C\in\mathbb{C}^{r\times n}_r$. Then
\begin{equation*}
A\{1,2,4^{\mathfrak{m}}\}=\left\{
   \begin{array}{c|c}
    C^{\sim}(CC^{\sim})^{-1}B^{-1}_{L} &B^{-1}_{L}\in\mathbb{C}^{r\times m}_{r} \\
   \end{array}
 \right\}.
\end{equation*}
\end{theorem}

\begin{remark}
 If $A\in\mathbb{C}^{m\times n}$ and $X\in A\{1,2\}$, we derive  \cite[Theorems 3 and  4]{minMPre} directly by Theorems \ref{m13ineqconditionth} and \ref{m14ineqconditionth}, respectively.
\end{remark}

\section{Characterizations of  the Minkowski inverse}\label{charactermimMPsec}
Based  on Lemma \ref{minMPexistconlem}, we start by proposing  several  different existence conditions of $A^{\mathfrak{m}}$ in the following theorem.
\begin{theorem}\label{mMPexitscondition}
Let $A\in\mathbb{C}^{m\times n}$. Then the following statements are equivalent:
\begin{enumerate}[$(1)$]
\item\label{minMPexistATSitem0} $A^{\mathfrak{m}}$ exists;
  \item\label{minMPexistATSitem1} ${\rm rank}(AA^{\sim})={\rm rank}(A^{\sim}A)={\rm rank}(A)$;
        \item\label{minMPexistATSitem3} ${\rm rank}({A^{\sim}}AA^{\sim})={\rm rank}(A)$;
  \item\label{minMPexistATSitem2} $A\mathcal{R}(A^{\sim})\oplus \mathcal{N}(A^{\sim})=\mathbb{C}^{m}$.
\end{enumerate}
\end{theorem}
\par
\begin{proof}
\eqref{minMPexistATSitem0} $\Leftrightarrow$ \eqref{minMPexistATSitem1}.  It is straightforward in view of  Lemma \ref{minMPexistconlem}. \par
\eqref{minMPexistATSitem1} $\Rightarrow$ \eqref{minMPexistATSitem3}. Since $\mathcal{R}(A)\cap \mathcal{N}(A^{\sim})=\{0\}$ from ${\rm rank}(A^{\sim}A)={\rm rank}(A)$, it follows from ${\rm rank}(AA^{\sim})={\rm rank}(A)$ that
\begin{align*}
{\rm rank}({A^{\sim}}AA^{\sim})&={\rm rank}(AA^{\sim})-{\rm dim}(\mathcal{R}(AA^{\sim})\cap \mathcal{N}(A^{\sim}))\\
&={\rm rank}(A)-{\rm dim}(\mathcal{R}(A)\cap \mathcal{N}(A^{\sim}))\\
&={\rm rank}(A).
\end{align*}
\par
\eqref{minMPexistATSitem3} $\Rightarrow$ \eqref{minMPexistATSitem2}.
From
\begin{equation*}
 {\rm rank}(A)={\rm rank}({A^{\sim}}AA^{\sim})={\rm rank}(AA^{\sim})-{\rm dim}(\mathcal{R}(AA^{\sim})\cap \mathcal{N}(A^{\sim}))
\end{equation*}
 and   ${\rm rank}(A)\geq {\rm rank}(AA^{\sim})$,
we have $\mathcal{R}(AA^{\sim})\cap \mathcal{N}(A^{\sim})=\{0\}$ and ${\rm rank}(AA^{\sim})={\rm rank}(A)$, which imply that $A\mathcal{R}(A^{\sim})\oplus \mathcal{N}(A^{\sim})=\mathbb{C}^{m}$.
\par
\eqref{minMPexistATSitem2} $\Rightarrow$ \eqref{minMPexistATSitem1}.
It follows from $A\mathcal{R}(A^{\sim})\oplus \mathcal{N}(A^{\sim})=\mathbb{C}^{m}$ that $\mathcal{R}(AA^{\sim})\cap \mathcal{N}(A^{\sim})=\{0\}$ and ${\rm rank}(AA^{\sim})={\rm rank}(A)$. Thus, $\mathcal{R}(A)\cap \mathcal{N}(A^{\sim})=\{0\}$, that is, ${\rm rank}(A^{\sim}A)={\rm rank}(A)$.
\end{proof}
\par
Subsequently, if $A^{\mathfrak{m}}$ exists for $A\in\mathbb{C}^{m\times n}$, applying Lemma \ref{ATS2Gchale} to
\eqref{minMPATS2} in Remark \ref{minMPATS2rem}, we directly obtain a new expression of the Minkowski inverse, $A^{\mathfrak{m}}=(A^{\sim}A)^{\#}A^{\sim}=A^{\sim}(AA^{\sim})^{\#}$, and
\begin{equation}\label{indAAeq}
{\rm Ind}(AA^{\sim})={\rm Ind}(A^{\sim}A)=1.
\end{equation}
However, for a matrix $A\in\mathbb{C}^{m\times n}$ satisfying  \eqref{indAAeq}, $A^{\mathfrak{m}}$  does not necessarily exist,  as will be shown in the following example.

\begin{example}
Let
\begin{equation*}
A= \left(
  \begin{array}{ccccc}
1&0&1&1\\0&0&1&0\\1&0&1&1\\0&0&0&0\\0&0&0&0\\
  \end{array}
\right).
\end{equation*}
It can be verified that ${\rm rank}(A^{\sim}AA^{\sim})={\rm rank}((A^{\sim}A)^{2})={\rm rank}(A^{\sim}A)={\rm rank}((AA^{\sim})^{2})={\rm rank}(AA^{\sim})=1$ and ${\rm rank}(A)=2$. Obviously, ${\rm Ind}(A^{\sim}A)={\rm Ind}(AA^{\sim})=1$, but ${\rm rank}(A^{\sim}AA^{\sim})\neq {\rm rank}(A)$, implying that $A^{\mathfrak{m}}$ does not exist.
\end{example}
\par
In the next theorem  we present  some   necessary and sufficient conditions for  the converse implication.
\begin{theorem}\label{IndchaAmMPth}
Let $A\in\mathbb{C}^{m\times n}$. Then the following statements are equivalent:
\begin{enumerate}[$(1)$]
  \item\label{IndchaAmMPit1} $A^{\mathfrak{m}}$ exists;
   \item\label{IndchaAmMPit2} ${\rm Ind}(A^{\sim}A)=1$ and $\mathcal{N}(A^{\sim}A)\subseteq \mathcal{N}(A)$;
  \item\label{IndchaAmMPit3}  ${\rm Ind}(AA^{\sim})=1$  and $\mathcal{R}(A)\subseteq \mathcal{R}(AA^{\sim}) $.
\end{enumerate}
\end{theorem}
\par
\begin{proof}
\eqref{IndchaAmMPit1} $\Leftrightarrow$ \eqref{IndchaAmMPit2}. ``$\Rightarrow$". It is obvious by Lemmas \ref{minMPexistconlem} and \ref{ATS2Gchale}.  ``$\Leftarrow$". Since ${\rm rank}(A)={\rm rank}(A^{\sim}A)$ from $\mathcal{N}(A^{\sim}A)\subseteq \mathcal{N}(A)$, it follows from ${\rm Ind}(A^{\sim}A)=1$ that $\mathcal{R}(A^{\sim})\cap \mathcal{N}(A)=\mathcal{R}(A^{\sim}A)\cap \mathcal{N}(A^{\sim}A)=\{0\}$, which implies that ${\rm rank}(AA^{\sim})={\rm rank}(A)$. Hence $A^{\mathfrak{m}}$ exists directly by Lemma \ref{minMPexistconlem}.\par
\eqref{IndchaAmMPit1} $\Leftrightarrow$ \eqref{IndchaAmMPit3}. Its proof is similar to that of \eqref{IndchaAmMPit1} $\Leftrightarrow$ \eqref{IndchaAmMPit2}.
\end{proof}
\par
As we all know,   Moore \cite{MooredefMPre}, Penrose \cite{penrosere}, and Desoer and Whalen \cite{DWdefMPre}  defined  the Moore-Penrose inverse  from different perspectives,  respectively. Next, we  review these definitions in the following lemma, and extend this result to the Minkowski inverse.

\begin{lemma}[Desoer-Whalen's and Moore's definitions  \cite{DWdefMPre,MooredefMPre}]
Let $A\in\mathbb{C}^{m\times n}$ and $X\in\mathbb{C}^{n\times m}$. Then the following statements are equivalent:
\begin{enumerate}[$(1)$]
  \item $X=A^{\dag}$;
  \item $XAa=a$ for $a\in\mathcal{R}({A^*})$, and $Xb=0$ for $b\in\mathcal{N}(A^*)$;
  \item $AX=P_{\mathcal{R}(A)}$, $XA=P_{\mathcal{R}(X)}$.
\end{enumerate}
\end{lemma}
 \par
There  is  an interesting   example showing that for some matrices $A\in\mathbb{C}^{m\times n}$ and $X\in\mathbb{C}^{n\times m}$,  $X$ may be not equivalent with  $A^{\mathfrak{m}}$ though $AX=P_{\mathcal{R}(A),\mathcal{N}(A^{\sim})}$ and  $XA=P_{\mathcal{R}(X),\mathcal{N}(A)}$.

\begin{example}\label{minMPmooredefce}
Let us consider the matrices
\begin{equation*}
  A=\left(
  \begin{array}{ccccc}
1&1&1&0&1\\0&1&0&1&0\\1&1&0&0&1\\0&0&0&0&0\\0&0&0&0&0
  \end{array}
\right),
X=\left(
  \begin{array}{ccccc}
0&-0.2&0.4&0&0\\0&0.4&0.2&0&0\\1&0&-1&0&0\\0&0.6&-0.2&0&0\\0&-0.2&0.4&0&0\\
  \end{array}
\right).
\end{equation*}
By calculation, we have ${\rm rank}(A^{\sim}AA^{\sim})={\rm rank}(A)=3$ and
\begin{equation*}
  A^{\mathfrak{m}}=\left(
  \begin{array}{ccccc}
  0&1&-2&0&0\\0&0&1&0&0\\1&0&-1&0&0\\0&1&-1&0&0\\0&-1&2&0&0\\
  \end{array}
\right).
\end{equation*}
Evidently, $X\neq A^{\mathfrak{m}}$. However, we can check  that  $X\in A\{1,2\}$ and $\mathcal{N}(X)=\mathcal{N}(A^{\sim})$, which imply that  $AX=P_{\mathcal{R}(A),\mathcal{N}(A^{\sim})}$ and  $XA=P_{\mathcal{R}(X),\mathcal{N}(A)}$.
\end{example}
\par
\begin{theorem}\label{threedefAmMPth}
Let $A\in\mathbb{C}^{m\times n}$ and $X\in\mathbb{C}^{n\times m}$. Then the following statements are equivalent:
\begin{enumerate}[$(1)$]
\item\label{threedefAmMPit3} $X=A^{\mathfrak{m}}$;
  \item\label{threedefAmMPit1} $XAa=a$ for $a\in\mathcal{R}({A^{\sim}})$, and $Xb=0$ for $b\in\mathcal{N}(A^{\sim})$;
  \item\label{threedefAmMPit2} $AX=P_{\mathcal{R}(A),\mathcal{N}(A^{\sim})}$, $XA=P_{\mathcal{R}(X),\mathcal{N}(A)}$ and $\mathcal{R}(X)\subseteq \mathcal{R}(A^{\sim})$.
\end{enumerate}
\end{theorem}
\par
\begin{proof}\eqref{threedefAmMPit3} $\Rightarrow$ \eqref{threedefAmMPit1}. It is obvious by Lemma \ref{mMPprole}. \par
\eqref{threedefAmMPit1} $\Rightarrow$ \eqref{threedefAmMPit2}.
It follows from $XAa=a$ for $a\in\mathcal{R}({A^{\sim}})$ that $XAA^{\sim}=A^{\sim}$, which shows that ${\rm rank}(A^{\sim})\leq {\rm rank}(X)$ and $\mathcal{R}(A^{\sim})\subseteq\mathcal{R}(X)$. And, from $Xb=0$ for $b\in\mathcal{N}(A^{\sim})$, we have $\mathcal{N}(A^{\sim})\subseteq\mathcal{N}(X)$, implying ${\rm rank}(X)\leq{\rm rank}(A^{\sim})$. Thus,  ${\rm rank}(X)={\rm rank}(A^{\sim})$, $\mathcal{R}(X)=\mathcal{R}(A^{\sim})$ and $\mathcal{N}(X)=\mathcal{N}(A^{\sim})$. Hence, again by $XAa=a$ for $a\in\mathcal{R}({A^{\sim}})=\mathcal{R}({X})$, we have $X\in A\{1,2\}$, which implies that the item \eqref{threedefAmMPit2} holds. \par
\eqref{threedefAmMPit2} $\Rightarrow$ \eqref{threedefAmMPit3}. Clearly, $AXA=P_{\mathcal{R}(A),\mathcal{N}(A^{\sim})}A=A$ and $XAX=P_{\mathcal{R}(X),\mathcal{N}(A)}X=X$, i.e., $X\in A\{1,2\}$. Then, from $AX=P_{\mathcal{R}(A),\mathcal{N}(A^{\sim})}$ and $\mathcal{R}(X)\subseteq \mathcal{R}(A^{\sim})$, we have $\mathcal{N}(X)=\mathcal{N}(A^{\sim})$ and $\mathcal{R}(X)= \mathcal{R}(A^{\sim})$. Hence, in view of
\eqref{minMPATS2} in Remark \ref{minMPATS2rem}, we see $X=A^{\mathfrak{m}}$.
\end{proof}
\par
A classic   characterization of the Moore-Penrose inverse proposed by  Bjerhammar is extended to the Minkowski inverse  in the following theorem.

\begin{theorem}\label{AmMPBjeharmmerth}
Let $A\in\mathbb{C}^{m\times n}$ with ${\rm rank}({A^{\sim}}AA^{\sim})={\rm rank}(A)$, and let $X\in\mathbb{C}^{n\times m}$. Then the following statements are equivalent:
\begin{enumerate}[$(1)$]
  \item\label{AmMPBjeharmmer1} $X=A^{\mathfrak{m}}$;
  \item\label{AmMPBjeharmmer2} There exist $B\in\mathbb{C}^{m\times m}$ and $C\in\mathbb{C}^{n\times n}$ such that  $AXA=A,X=A^{\sim}B,X=CA^{\sim}$.
\end{enumerate}
Moreover,
\begin{align*}
 B&=(A^{\sim})^{(1)}A^{\mathfrak{m}}+(I_m-(A^{\sim})^{(1)}A^{\sim})Y,\\ C&=A^{\mathfrak{m}}(A^{\sim})^{(1)}+Z(I_n-A^{\sim}(A^{\sim})^{(1)}),
\end{align*}
where $Y\in\mathbb{C}^{m\times m}$ and $Z\in\mathbb{C}^{n\times n}$ are arbitrary, and $(A^{\sim})^{(1)}\in (A^{\sim})\{1\}$.
\end{theorem}
\par
\begin{proof}
It is easily obtained in  terms of Remark \ref{minMPATS2rem} and  Lemma \ref{AXDsolle}.
\end{proof}

\begin{corollary}
Let $A\in\mathbb{C}^{m\times n}$ with ${\rm rank}({A^{\sim}}AA^{\sim})={\rm rank}(A)$, and let $X\in\mathbb{C}^{n\times m}$. Then the following statements are equivalent:
\begin{enumerate}[$(1)$]
  \item $X=A^{\mathfrak{m}}$;
  \item There exists $D\in\mathbb{C}^{m\times n}$ such that  $AXA=A,X=A^{\sim}DA^{\sim}$.
\end{enumerate}
In this case,
\begin{equation*}
  D=(A^{\sim})^{(1)}A^{\mathfrak{m}}(A^{\sim})^{(1)}+(I_m-(A^{\sim})^{(1)}A^{\sim})Y+Z(I_n-A^{\sim}(A^{\sim})^{(1)}),
\end{equation*}
where $Y,Z\in\mathbb{C}^{m\times n}$ are arbitrary, and $A^{(1)}\in A\{1\}$.
\end{corollary}
\par
\begin{proof}
It is a direct corollary of Theorem \ref{AmMPBjeharmmerth}.
\end{proof}

\section{Further characterizations of  the Minkowski inverse} \label{characterminMPfursec}
As it has been stated in Section \ref{introductionsection}, a great deal of mathematical effort \cite{ringbookre,Centralizerref,syl} has been devoted to the study of the Moore-Penrose inverse in a ring with   involution. It is  observed that $\mathbb{C}^{m\times n}$  is not a ring or even a semigroup for matrix multiplication (unless $m=n$).
However, notice two interesting facts. One is that an involution \cite{Centralizerref} $a \mapsto a^*$ in a ring $R$ is a map from $R$ to $R$ such that $(a^*)^*=a$, $(a+b)^*=a^*+b^*$, and $(ab)^*=b^*a^*$ for all $a,b\in R$, and  the other one is that the Minkowski adjoint $A^{\sim}$ has  similar properties, that is, $(A^{\sim})^{\sim}=A$, $(A+C)^{\sim}=A^{\sim}+C^{\sim}$, and $(AB)^{\sim}=B^{\sim}A^{\sim}$,  where $A,C\in\mathbb{C}^{m\times n}$, and $B\in\mathbb{C}^{n\times l}$.
Based on the above considerations, the purpose of this section is to extend some characterizations of the Moore-Penrose inverse in rings, mainly mentioned in \cite{Centralizerref,syl}, to  the Minkowski inverse.  Inspired by \cite[Theorem 3.12, Corollary 3.17]{Centralizerref}, we give the following two results in the first part of this section.

\begin{theorem}\label{XYchaAmMPth}
Let $A\in\mathbb{C}^{m\times n}$. Then the following statements are equivalent:
\begin{enumerate}[$(1)$]
  \item\label{XYchaAmMPit1} $A^{\mathfrak{m}}$ exists;
  \item\label{XYchaAmMPit2} There exists $X\in \mathbb{C}^{m\times m}$ such that $A=XAA^{\sim}A$;
  \item\label{XYchaAmMPit3} There exists $Y\in \mathbb{C}^{n\times n}$ such that $A=AA^{\sim}AY$.
\end{enumerate}
      In this case, $A^{\mathfrak{m}}=(XA)^{\sim}=(AY)^{\sim}$.
\end{theorem}

\begin{proof}
It is easy to see that there exists $X\in \mathbb{C}^{m\times m}$ such that $A=XAA^{\sim}A$ if and only if $\mathcal{N}(AA^{\sim}A)\subseteq \mathcal{N}(A)$ if and only if $ {\rm rank}(A)={\rm rank}(AA^{\sim}A)$.
Then the equivalence of \eqref{XYchaAmMPit1}  and  \eqref{XYchaAmMPit2} is obvious by the item  \eqref{minMPexistATSitem3} in Theorem  \ref{mMPexitscondition}. And, the proof of the   equivalence of \eqref{XYchaAmMPit1}  and  \eqref{XYchaAmMPit3} can be completed by the method analogous to that used above. \par
Moreover, if $A^{\mathfrak{m}}$ exists,  we first claim that $(XA)^{\sim}\in A\{1,3^{\mathfrak{m}},4^{\mathfrak{m}}\}$. In fact, using $A=XAA^{\sim}A$ we infer that
\begin{align*}
(A(XA)^{\sim})^{\sim}&=XAA^{\sim}=XA(XAA^{\sim}A)^{\sim}
=XAA^{\sim}AA^{\sim}X^{\sim}=A(XA)^{\sim},\\
A(XA)^{\sim}A&=(A(XA)^{\sim})^{\sim}A=XAA^{\sim}A=A,\\
((XA)^{\sim}A)^{\sim}&=(A^{\sim}X^{\sim}A)^{\sim}=((XAA^{\sim}A)^{\sim}X^{\sim}A)^{\sim}=(A^{\sim}AA^{\sim}(X^{\sim})^{2}A
)^{\sim}\\&=(A^{\sim}XAA^{\sim}AA^{\sim}(X^{\sim})^2A)^{\sim}
=(A^{\sim}XXAA^{\sim}AA^{\sim}AA^{\sim}(X^{\sim})^2A)^{\sim}\\
&=(A^{\sim}(X)^2(AA^{\sim})^3(X^{\sim})^2A)^{\sim}
= A^{\sim}(X)^2(AA^{\sim})^3(X^{\sim})^2A=(XA)^{\sim}A,
\end{align*}
which imply that $(XA)^{\sim}\in A\{1,3^{\mathfrak{m}},4^{\mathfrak{m}}\}$. Finally, according to \eqref{minMPA13mA14meq},  we obtain that
\begin{align*}
A^{\mathfrak{m}}&=(XA)^{\sim}A(XA)^{\sim}=((XA)^{\sim}A)^{\sim}(XA)^{\sim}=A^{\sim}XAA^{\sim}X^{\sim}\\
&=(A(XA)^{\sim}A)^{\sim}X^{\sim}=(XA)^{\sim}.
\end{align*}
Using the same way as in the above proof, we can  carry out the proof of    $(AY)^{\sim}\in A\{1,3^{\mathfrak{m}},4^{\mathfrak{m}}\}$  and $A^{\mathfrak{m}}=(AY)^{\sim}$.
\end{proof}
\par
A well-known result is given directly in the following lemma, which will be useful in the proof of the next theorem.

\begin{lemma}\label{mattacobson}
Let $A\in\mathbb{C}^{m\times n}$ and $B\in\mathbb{C}^{n\times m}$. Then $I_m-AB$ is nonsingular if and only if $I_n-BA$ is nonsingular, in which case, $(I_m-AB)^{-1}=I_m+A(I_n-BA)^{-1}B$.
\end{lemma}

\begin{theorem}\label{inverchamMPth}
Let $A\in\mathbb{C}^{m\times n}$ and  $A^{(1)}\in A\{1\}$. Then the following statements are equivalent:
\begin{enumerate}[$(1)$]
  \item\label{inverchamMPit1} $A^{\mathfrak{m}}$ exists;
    \item\label{inverchamMPit3} $A^{\sim}A+I_{n}-A^{(1)}A$ is nonsingular;
  \item\label{inverchamMPit2} $AA^{\sim}+I_{m}-AA^{(1)}$ is nonsingular.
\end{enumerate}
In this case,
\begin{align*}
 A^{\mathfrak{m}}&=(A(A^{\sim}A+I_{n}-A^{(1)}A)^{-1})^{\sim}\\
 &=((AA^{\sim}+I_{m}-AA^{(1)})^{-1}A)^{\sim}.
\end{align*}
\end{theorem}

\begin{proof}
Denote $B=A^{\sim}A+I_{n}-A^{(1)}A$ and $C=AA^{\sim}+I_{m}-AA^{(1)}$. \par
\eqref{inverchamMPit1} $\Rightarrow$ \eqref{inverchamMPit3}. If $A^{\mathfrak{m}}$ exists,  using   items \eqref{XYchaAmMPit1} and \eqref{XYchaAmMPit2} in Theorem \ref{XYchaAmMPth},  we have $A=XAA^{\sim}A$ for some $X\in\mathbb{C}^{m\times m}$. It can be easily verified that
\begin{equation*}
 (A^{(1)}XA+I_{n}-A^{(1)}A)(A^{(1)}AA^{\sim}A+I_n-A^{(1)}A)=I_{n},
\end{equation*}
which  shows the nonsingularity of $D:=A^{(1)}AA^{\sim}A+I_n-A^{(1)}A$. And, $D$ can be rewritten as $D=I_n-A^{(1)}A(I_n-A^{\sim}A)$. Thus, by  Lemma \ref{mattacobson}, it is easy to see that $B$ is nonsingular. \par
\eqref{inverchamMPit3} $\Rightarrow$ \eqref{inverchamMPit1}.
Since $B$ is nonsingular, from $AB=AA^{\sim}A$ we have that $A=AA^{\sim}AB^{-1}$.
Therefore, $A^{\mathfrak{m}}$ exists by  items \eqref{XYchaAmMPit1} and \eqref{XYchaAmMPit3} in  Theorem \ref{XYchaAmMPth}.\par
\eqref{inverchamMPit2} $\Leftrightarrow$ \eqref{inverchamMPit3}. Since $B$  and $C$ can be rewritten as  $B=I_n-(A^{(1)}-A^{\sim})A$ and $C=I_{m}-A(A^{(1)}-A^{\sim})$, from Lemma \ref{mattacobson} we  have the equivalence of \eqref{inverchamMPit2} and \eqref{inverchamMPit3} immediately.\par
In this case, from items \eqref{mMPRN} and \eqref{AAmMPeq} in Lemma \ref{mMPprole},  we infer that
\begin{align*}
B^{\sim}A^{\mathfrak{m}}&=(A^{\sim}A+I_{n}-A^{(1)}A)^{\sim}A^{\mathfrak{m}}\\
&=A^{\sim}AA^{\mathfrak{m}}+A^{\mathfrak{m}}-A^{\sim}(A^{\sim})^{(1)}A^{\mathfrak{m}}\\
&=A^{\sim},
\end{align*}
which, together with the item \eqref{inverchamMPit3}, gives $A^{\mathfrak{m}}=(AB^{-1})^{\sim}$. Analogously, we can derive  $A^{\mathfrak{m}}=(C^{-1}A)^{\sim}$. This completes the proof.
\end{proof}
\par
The Sylvester matrix equation  \cite{sylorignre} has numerous applications in neural networks, robust control, graph theory, and other areas of system and control theory.
Motivated by \cite[Theorem 2.3]{syl}, in the following theorem, we use the solvability of a certain Sylverster matrix equation to characterize the existence of  the Minkowski inverse, and apply its solutions to represent the Minkowski inverse.

\begin{theorem}\label{SylvestermMPth}
Let $A\in\mathbb{C}^{m\times n}$. Then the following statements are equivalent:
\begin{enumerate}[$(1)$]
  \item\label{SylvestermMPitem1} $A^{\mathfrak{m}}$ exists;
  \item\label{SylvestermMPitem2} ${\rm rank}(AA^{\sim})={\rm rank}(A^{\sim})$ and there exist $X\in\mathbb{C}^{m\times m}$ and a projector $Y\in\mathbb{C}^{m\times m}$ such that
      \begin{equation}\label{XAAYAEq}
      XAA^{\sim}-YX=I_m,
      \end{equation}
 $AA^{\sim}X=XAA^{\sim}$ and $AA^{\sim}Y=0$.
\end{enumerate}
In this case,
\begin{equation}\label{MinMPAeAwXeq}
A^{\mathfrak{m}}=A^{\sim}X.
\end{equation}
\end{theorem}

\begin{proof}
\eqref{SylvestermMPitem1} $\Rightarrow$ \eqref{SylvestermMPitem2}.
If $A^{\mathfrak{m}}$ exists, it is clear by Lemma \ref{minMPexistconlem} to see that ${\rm rank}(AA^{\sim})={\rm rank}(A^{\sim})$. Let $Q=AA^{\sim}+I_m-AA^{\mathfrak{m}}$.
By items \eqref{mMPRN} and \eqref{AAmMPeq} in Lemma \ref{mMPprole}, it is easily be verified that $Q((A^{\sim})^{\mathfrak{m}}A^{\mathfrak{m}}+I_m-AA^{\mathfrak{m}})=I_m$, showing   the nonsingularity of $Q$.
And, $AA^{\mathfrak{m}}Q=QAA^{\mathfrak{m}}=AA^{\sim}$.
Denote $Y=I_m-AA^{\mathfrak{m}}$. Clearly, $Y^2=Y$ and $AA^{\sim}Y=YAA^{\sim}=0$.
Let $X=AA^{\mathfrak{m}}Q^{-1}-Y$. Hence,
\begin{align*}
XAA^{\sim}&=(AA^{\mathfrak{m}}Q^{-1}-Y)AA^{\mathfrak{m}}Q =AA^{\mathfrak{m}}Q^{-1}QAA^{\mathfrak{m}}=AA^{\mathfrak{m}},\\
AA^{\sim}X&=QAA^{\mathfrak{m}}(AA^{\mathfrak{m}}Q^{-1}-Y) =QAA^{\mathfrak{m}}Q^{-1}=AA^{\mathfrak{m}}QQ^{-1}=AA^{\mathfrak{m}},\\
-YX&=-Y(AA^{\mathfrak{m}}Q^{-1}-Y)=-YAA^{\mathfrak{m}}Q^{-1}+Y=Y.
\end{align*}
Evidently, $XAA^{\sim}=AA^{\sim}X$ and $XAA^{\sim}-YX=I_m$. \par
\eqref{SylvestermMPitem2} $\Rightarrow$ \eqref{SylvestermMPitem1}.  Premultiplying \eqref{XAAYAEq}  by $AA^{\sim}$, we have
$AA^{\sim}XAA^{\sim}-AA^{\sim}YX=AA^{\sim}$,  which, together with $AA^{\sim}X=XAA^{\sim}$ and $AA^{\sim}Y=0$,  yields that $AA^{\sim}AA^{\sim}X=AA^{\sim}$ if and only if $\mathcal{R}(AA^{\sim}X-I_m)\subseteq \mathcal{N}(AA^{\sim})$. Since $\mathcal{N}(AA^{\sim})=\mathcal{N}(A^{\sim})$ from ${\rm rank}(AA^{\sim})={\rm rank}(A^{\sim})$, we get $A^{\sim}=A^{\sim}AA^{\sim}X$, i.e.,
\begin{equation}\label{AAAXeA}
A=X^{\sim}AA^{\sim}A.
\end{equation}
Consequently, $A^{\mathfrak{m}}$ exists in terms of  items \eqref{inverchamMPit1} and \eqref{inverchamMPit3} in Theorem \ref{XYchaAmMPth}.\par
Finally, if $A^{\mathfrak{m}}$ exists, applying Theorem \ref{XYchaAmMPth} to \eqref{AAAXeA}  we have \eqref{MinMPAeAwXeq} directly.
\end{proof}

\begin{remark}
Let $A\in\mathbb{C}^{m\times n}$. Using a easy result that $A^{\mathfrak{m}}$ exists if and only if $(A^{\sim})^{\mathfrak{m}}$ exists, by Theorem \ref{SylvestermMPth} we conclude that the following statements are equivalent:
\begin{enumerate}[$(1)$]
  \item $A^{\mathfrak{m}}$ exists;
  \item ${\rm rank}(A^{\sim}A)={\rm rank}(A)$ and there exist $X\in\mathbb{C}^{n\times n}$ and a projector $Y\in\mathbb{C}^{n\times n}$ such that $XA^{\sim}A-YX=I_n$,
 $A^{\sim}AX=XA^{\sim}A$ and $A^{\sim}AY=0$.
\end{enumerate}
In this case, $A^{\mathfrak{m}}=(AX)^{\sim}$.
\end{remark}

\section{Characterizing Minkowski inverse by rank equation}\label{characterrankminMPsec}
It is well known that the Hartwig-Spindelb\"{o}ck decomposition is an effective and basic tool for finding representations of various generalized inverses and matrix classes (see \cite{BakHSdec,HSdec}).
A new condition for the existence of the Minkowski inverse is given by the Hartwig-Spindelb\"{o}ck decomposition in this section. Under this condition, we present a new representation of the Minkowski inverse. We first  introduce the following notations used in the section.
 \par
For $A\in\mathbb{C}^{n\times n}$ given by \eqref{HSAdec} in   Lemma \ref{HSth}, let
\begin{equation*}
\left(
           \begin{array}{cc}
             G_1 & G_2 \\
             G_3 & G_4 \\
           \end{array}
         \right)= U^* GU,
 \end{equation*}
 where $G_1\in\mathbb{C}^{r\times r}$, $G_2\in\mathbb{C}^{r\times (n-r)}$, $G_3\in\mathbb{C}^{(n-r)\times r}$ and $G_4\in\mathbb{C}^{(n-r)\times (n-r)}$, and let
 \begin{equation*}
 \Delta= \left(
            \begin{array}{cc}
              K & L \\
            \end{array}
          \right)
         U^* GU
         \left(
            \begin{array}{c}
              K^* \\ L^* \\
            \end{array}
          \right).
 \end{equation*}
\begin{theorem}\label{HSmMpdecth}
Let $A$ be given in \eqref{HSAdec}.
\begin{enumerate}[$(1)$]
  \item\label{HSmMpdecitem1} ${\rm rank}(A)={\rm rank}(AA^{\sim})$ if and only if $ \Delta$  is nonsingular.
  \item\label{HSmMpdecitem2} ${\rm rank}(A)={\rm rank}(A^{\sim}A)$ if and only if $G_1$ is nonsingular.
 \item\label{HSmMpdecitem3} If $ \Delta$  and $G_1$ are nonsingular, then
 \begin{align}
A^{\mathfrak{m}}&=GU\left(
                     \begin{array}{cc}
                       K^*(G_1\Sigma\Delta)^{-1} & 0 \\
                       L^*(G_1\Sigma\Delta)^{-1} & 0 \\
                     \end{array}
                   \right)U^*G\label{HSmMpdec}\\
        & =   U\left(
            \begin{array}{cc}
            (G_1K^*+G_2L^*)(\Sigma\Delta)^{-1}& (G_1K^*+G_2L^*)(G_1\Sigma\Delta)^{-1}G_2 \\
             (G_3K^*+G_4L^*)(\Sigma\Delta)^{-1}      & (G_3K^*+G_4L^*)(G_1\Sigma\Delta)^{-1}G_2 \\
          \end{array}
          \right)U^*.\label{HSmMpdecUU}
\end{align}
\end{enumerate}
\end{theorem}

\begin{proof}
\eqref{HSmMpdecitem1}. Using the Hartwig-Spindelb\"{o}ck decomposition, we have \begin{align*}
{\rm rank}(A)={\rm rank}(AA^{\sim}) &\Leftrightarrow {\rm rank}(A)= {\rm rank}\left(
   \left(
     \begin{array}{cc}
       \Sigma K & \Sigma L \\
       0 & 0 \\
     \end{array}
   \right)
U^*GU
   \left(
     \begin{array}{cc}
       (\Sigma K)^* & 0 \\
       (\Sigma L)^* & 0 \\
     \end{array}
   \right)
   \right)\\
 &\Leftrightarrow  {\rm rank}(A)={\rm rank}\left(
 \left(
   \begin{array}{cc}
        K    &    L \\
   \end{array}
 \right)
U^*GU
 \left(
   \begin{array}{c}
    K^*  \\
     L^* \\
   \end{array}
 \right)
 \right),
\end{align*}
which is equivalent to that $\Delta$  is nonsingular.
\par
\eqref{HSmMpdecitem2}.
Since $\left(
  \begin{array}{cc}
    \Sigma K & \Sigma L\\
  \end{array}
\right)$ is of full row rank by \eqref{HSKLconditioneq},
using again the Hartwig-Spindelb\"{o}ck decomposition we derive   that
\begin{align*}
{\rm rank}(A)={\rm rank}(A^{\sim}A) &\Leftrightarrow
{\rm rank}(A) ={\rm rank}\left(
   \left(
     \begin{array}{cc}
       (\Sigma K)^* & 0 \\
       (\Sigma L)^* & 0 \\
     \end{array}
   \right)\left(
           \begin{array}{cc}
             G_1 & G_2 \\
             G_3 & G_4 \\
           \end{array}
         \right)
         \left(
     \begin{array}{cc}
       \Sigma K & \Sigma L \\
       0 & 0 \\
     \end{array}
   \right)
\right)\\
&\Leftrightarrow
{\rm rank}(A) ={\rm rank}\left(
\left(
  \begin{array}{cc}
    \Sigma K & \Sigma L\\
  \end{array}
\right)^*G_1
\left(
  \begin{array}{cc}
    \Sigma K & \Sigma L\\
  \end{array}
\right)
\right)\\
&\Leftrightarrow
{\rm rank}(A)={\rm rank}(G_1),
\end{align*}
which is equivalent to that $G_1$ is nonsingular.
\par
\eqref{HSmMpdecitem3}.
Note that $A$ given in \eqref{HSAdec} can be rewritten as
\begin{equation}\label{AHSdecrefullrankeq}
A=U\left(
     \begin{array}{c}
        \Sigma \\
       0 \\
     \end{array}
   \right)
   \left(\begin{array}{cc}
                K & L \\
            \end{array}
          \right)U^*,
\end{equation}
where
$B:=U\left(\begin{array}{c}
        \Sigma \\0 \\
     \end{array}
   \right)$
   and
$C:=\left(\begin{array}{cc}
                K & L \\
            \end{array}
          \right)U^*$
are of  full column rank and full  row rank, respectively.
If $ \Delta$  and $G_1$ are nonsingular, by items \eqref{HSmMpdecitem1} and \eqref{HSmMpdecitem2} and Lemma \ref{minMPexistconlem} we see that $A^{\mathfrak{m}}$ exists.
Therefore, applying Lemma \ref{minMPfullreprelemma} to \eqref{AHSdecrefullrankeq} yields that
 \begin{align*}
  A^{\mathfrak{m}}= & C^{\sim}(CC^{\sim})^{-1}(B^{\sim}B)^{-1}B^{\sim}\\
    =& GU\left(
     \begin{array}{c}
        K^* \\
       L^* \\
     \end{array}
   \right)G\left(
   \left(\begin{array}{cc}
                K & L \\
            \end{array}
          \right)U^*GU\left(
     \begin{array}{c}
        K^* \\
       L^* \\
     \end{array}
   \right)G \right)^{-1}
   \\&
   \left(G\left(\begin{array}{cc}
               {\Sigma} & 0 \\
            \end{array}
          \right)U^*G
          U\left(\begin{array}{c}
        \Sigma \\0 \\
     \end{array}
   \right)
   \right)^{-1}G\left(\begin{array}{cc}
               {\Sigma} & 0 \\
            \end{array}
          \right)U^*G\\
  =&GU\left(
     \begin{array}{c}
        K^* \\
       L^* \\
     \end{array}
   \right){\Delta}^{-1}(\Sigma G_1 \Sigma)^{-1}\left(\begin{array}{cc}
               {\Sigma} & 0 \\
            \end{array}
          \right)U^*G\\
   =&GU\left(
                     \begin{array}{cc}
                       K^*(G_1\Sigma\Delta)^{-1} & 0 \\
                       L^*(G_1\Sigma\Delta)^{-1} & 0 \\
                     \end{array}
                   \right)U^*G\\
   =&U\left(
           \begin{array}{cc}
             G_1 & G_2 \\
             G_3 & G_4 \\
           \end{array}
         \right)
         \left(
            \begin{array}{cc}
                  K^*(G_1\Sigma\Delta)^{-1} & 0 \\
                 L^*(G_1\Sigma\Delta)^{-1} & 0 \\
          \end{array}
          \right)
           \left(
           \begin{array}{cc}
             G_1 & G_2 \\
             G_3 & G_4 \\
           \end{array}
         \right)U^*\\
       =& U\left(
            \begin{array}{cc}
            (G_1K^*+G_2L^*)(\Sigma\Delta)^{-1}& (G_1K^*+G_2L^*)(G_1\Sigma\Delta)^{-1}G_2 \\
             (G_3K^*+G_4L^*)(\Sigma\Delta)^{-1}      & (G_3K^*+G_4L^*)(G_1\Sigma\Delta)^{-1}G_2 \\
          \end{array}
          \right)U^*,
 \end{align*}
 which completes the proof of this theorem.
\end{proof}

\begin{example}
In order to illustrate Theorem \ref{HSmMpdecth}, let us consider the matrix $A$ given in Example \ref{minMPmooredefce}. Then the Hartwig-Spindelb\"{o}ck decomposition of $A$ is
\begin{equation*}
  A=U\left(
     \begin{array}{cc}
       \Sigma K & \Sigma L \\
       0 & 0 \\
     \end{array}
   \right)U^*,
\end{equation*}
where
\begin{align*}
  U&=\left(
     \begin{array}{ccccc}
-0.73056&0.27137&-0.62661&0&0\\-0.27429&-0.95698&-0.094654&0&0\\-0.62534&0.10272&0.77356&0&0\\0&0&0&0&1\\0&0&0&1&0\\
     \end{array}
   \right),
\Sigma =\left(
          \begin{array}{ccc}
2.635&0&0\\0&1.2685&0\\0&0&0.66897\\
          \end{array}
        \right),\\
        K&=\left(
          \begin{array}{ccc}
0.71899&0.42393&0.16652\\-0.22319&0.54174&0.024188\\0.40383&-0.11142&-0.86962\\
          \end{array}
        \right),
L=\left(
          \begin{array}{ccc}
-0.51457&-0.10409\\0.29491&-0.75442\\0.21966&-0.14149\\
          \end{array}
        \right).
\end{align*}
And, we have that
\begin{equation*}
  G_1=
  \left(
    \begin{array}{ccc}
  0.06743&-0.3965&0.91556\\-0.3965&-0.85272&-0.34009 \\0.91556&-0.34009&-0.21471\\
    \end{array}
  \right),
  \Delta=
   \left(
    \begin{array}{ccc}
-0.47044&-0.3035&-0.22606\\-0.3035&-0.82606&0.12956\\-0.22606&0.12956&-0.9035\\
    \end{array}
  \right).
\end{equation*}
Thus, it is easy to check that ${\rm rank}(G_1)={\rm rank}(\Delta)=3$. Moreover, $A^{\mathfrak{m}}$ calculated by \eqref{HSmMpdec} or \eqref{HSmMpdecUU} is the same as that in Example \ref{minMPmooredefce}, so it is omitted.
\end{example}
\par
Gro{\ss} \cite{solution} considered an interesting problem what characterizations of $B$ and $C$ are when $X=A^{\dag}$ is assumed to be the unique solution of \eqref{rankeq} in Lemma \ref{raneconle}.
This issue was once more revisited by \cite{solutionDinversere} and \cite{solutionCoreinversere} on the Drazin inverse and the core inverse, respectively.
Subsequently, we apply Theorem \ref{HSmMpdecth} to provide another  characterization  of the Minkowski inverse.

\begin{theorem}\label{rankmMPth1}
Let $A$ be given in \eqref{HSAdec} with ${\rm rank}(A^{\sim}AA^{\sim})={\rm rank}(A)$, and let $X\in\mathbb{C}^{n\times n}$. Then $X=A^{\mathfrak{m}}$ is the unique solution of  the rank equation   \eqref{rankeq}
if and only if
\begin{equation}\label{BandCdec}
B=U\left(
     \begin{array}{cc}
       B_1 & B_2 \\
       0 & 0 \\
     \end{array}
   \right)
U^*G \text{ and  }
C=GUTU^*,
\end{equation}
where
\begin{align*}
     T&=\left(
                \begin{array}{cc}
                  J_1\Sigma K &  J_1\Sigma L \\
                   J_3\Sigma K &  J_3\Sigma L \\
                \end{array}
              \right),\\
B_1&=\left(
       \begin{array}{cc}
         \Sigma K& \Sigma L \\
       \end{array}
     \right)
     \left[
       \begin{array}{c}
         {T}^{(1)}\left(
                         \begin{array}{c}
                           K^*(G_1\Sigma\Delta)^{-1} \\
                            L^*(G_1\Sigma\Delta)^{-1}  \\
                         \end{array}
                       \right)\\
       \end{array}+(I_n-T^{(1)}T)Y_1
     \right],\\
     B_2&=\left(
       \begin{array}{cc}
         \Sigma K& \Sigma L \\
       \end{array}
     \right)(I_n-T^{(1)}T)Y_2,
\end{align*}
where $J_1\in\mathbb{C}^{r\times r}$ and $J_3\in\mathbb{C}^{(n-r)\times r}$ satisfy $\mathcal{N}(T^*)\subseteq \mathcal{N} \left(\left(\begin{array}{cc} K & L \\ \end{array}\right)\right)$,  $Y_1\in\mathbb{C}^{n\times r}$ and $Y_2\in\mathbb{C}^{n\times (n-r)}$ are arbitrary, and ${T}^{(1)}\in T\{1\}$.
\end{theorem}

\begin{proof}
We first prove the ``only if" part. If $X=A^{\mathfrak{m}}$ is the unique solution of   \eqref{rankeq},  from Lemma \ref{raneconle} we have
$ B=AH$ and $ C=JA$
for some $H,J\in\mathbb{C}^{n\times n}$. Put
\begin{equation*}
\left(
          \begin{array}{cc}
            H_1 & H_2 \\
            H_3 & H_4 \\
          \end{array}
        \right)= U^*HGU,
\left(
          \begin{array}{cc}
            J_1 & J_2 \\
            J_3 & J_4 \\
          \end{array}
        \right)= U^*GJU,
\end{equation*}
where $H_1,J_1\in\mathbb{C}^{r\times r}$, $H_2,J_2\in\mathbb{C}^{r\times (n-r)}$, $H_3,J_3\in\mathbb{C}^{(n-r)\times r}$, and $H_4,J_4\in\mathbb{C}^{(n-r)\times (n-r)}$.  Thus
\begin{equation}\label{BHSdec}
B=AH=U\left(
          \begin{array}{cc}
            \Sigma KH_1+\Sigma L H_3 &  \Sigma KH_2+\Sigma L H_4 \\
            0 & 0 \\
          \end{array}
        \right)U^*G,
\end{equation}
\begin{equation}\label{CHSdec}
C=JA=GU\left(
          \begin{array}{cc}
          J_1 \Sigma K& J_1 \Sigma L \\
            J_3   \Sigma K &   J_3   \Sigma L \\
          \end{array}
        \right)U^*.
\end{equation}
Note that \cite[Formula (1.4)]{BakHSdec} has  shown  that
\begin{equation}\label{AMPHSdec}
A^{\dag}=U\left(
          \begin{array}{cc}
            K^*\Sigma^{-1} & 0 \\
            L^*\Sigma^{-1} & 0 \\
          \end{array}
        \right)U^*.
\end{equation}
Then inserting  \eqref{BHSdec}, \eqref{CHSdec} and \eqref{AMPHSdec} to \eqref{ranksolution} gives
\begin{equation}\label{XHSdec}
X=GU
\left(
  \begin{array}{cc}
    J_1\Sigma K H_1+J_1\Sigma L H_3 &   J_1\Sigma K H_2+J_1\Sigma L H_4 \\
      J_3\Sigma K H_1+J_3\Sigma L H_3  &  J_3\Sigma K H_2+J_3\Sigma L H_4  \\
  \end{array}
\right)
U^*G.
\end{equation}
By a comparison of  \eqref{HSmMpdec} in Theorem  \ref{HSmMpdecth} with  \eqref{XHSdec}, we see that
\begin{align*}
X=A^{\mathfrak{m}}&\Leftrightarrow{\left\{ {\begin{array}{*{20}{l}}
 J_3\Sigma K H_1+J_3\Sigma L H_3 =   L^*(G_1\Sigma\Delta)^{-1},\\
 J_3\Sigma K H_1+J_3\Sigma L H_3=   L^*(G_1\Sigma\Delta)^{-1},\\
 J_1\Sigma K H_2+J_1\Sigma L H_4=0,\\
   J_3\Sigma K H_2+J_3\Sigma L H_4=0,
\end{array}} \right.}
\end{align*}
which can be rewritten as
\begin{equation}\label{H1H3H2H4}
T\left(
          \begin{array}{c}
                H_1 \\
            H_3 \\
          \end{array}
        \right)=\left(
          \begin{array}{c}
                K^*(G_1\Sigma\Delta)^{-1} \\
           L^*(G_1\Sigma\Delta)^{-1} \\
          \end{array}
        \right),
    T\left(
          \begin{array}{c}
                H_2 \\
            H_4 \\
          \end{array}
        \right)=  \left(
          \begin{array}{c}
               0 \\
           0 \\
          \end{array}
        \right),
\end{equation}
where $T=\left(
                \begin{array}{cc}
                  J_1\Sigma K &  J_1\Sigma L \\
                   J_3\Sigma K &  J_3\Sigma L \\
                \end{array}
              \right)$.
Applying Lemma \ref{AXDsolle} to \eqref{H1H3H2H4}, we conclude that
$J_1\in\mathbb{C}^{r\times r}$ and $J_3\in\mathbb{C}^{(n-r)\times r}$ satisfy
\begin{align*}
TT^{(1)}\left(
                         \begin{array}{c}
                           K^*(G_1\Sigma\Delta)^{-1} \\
                            L^*(G_1\Sigma\Delta)^{-1}  \\
                         \end{array}
                       \right)=
                       \left(
                         \begin{array}{c}
                           K^*(G_1\Sigma\Delta)^{-1} \\
                            L^*(G_1\Sigma\Delta)^{-1}  \\
                         \end{array}
                       \right)
                       &\Leftrightarrow
                       \mathcal{R}\left(\left(
                         \begin{array}{c}
                           K^*(G_1\Sigma\Delta)^{-1} \\
                            L^*(G_1\Sigma\Delta)^{-1}  \\
                         \end{array}
                       \right)\right) \subseteq \mathcal{R}(T)\\
           &  \Leftrightarrow
            \mathcal{R}\left(\left(
                         \begin{array}{c}
                           K^* \\
                            L^*  \\
                         \end{array}
                       \right)\right) \subseteq \mathcal{R}(T) \\
                          &  \Leftrightarrow
        \mathcal{N}(T^*)    \subseteq  \mathcal{N}\left(\left(
                         \begin{array}{cc}
                           K & L
                         \end{array}
                       \right)\right),
\end{align*}
 and
\begin{align}
\left(
  \begin{array}{c}
    H_1 \\
    H_3 \\
  \end{array}
\right)&=
T^{(1)}
\left(
  \begin{array}{c}
 K^*(G_1\Sigma\Delta){-1} \\
  L^*(G_1\Sigma\Delta){-1}  \\
  \end{array}
\right)
+(I_n-T^{(1)}T)Y_1,\label{B1HSdec}
\\
\left(
  \begin{array}{c}
    H_2 \\
    H_4 \\
  \end{array}
\right)&=(I_n-T^{(1)}T)Y_2,\label{B2HSdec}
\end{align}
where $Y_1\in\mathbb{C}^{n\times r}$ and $Y_2\in\mathbb{C}^{n\times (n-r)}$ are arbitrary, and ${T}^{(1)}\in T\{1\}$. Hence, premultiplying \eqref{B1HSdec} and \eqref{B2HSdec} by
$\left(
       \begin{array}{cc}
         \Sigma K& \Sigma L \\
       \end{array}
     \right)$,
from   \eqref{BHSdec} and \eqref{CHSdec} we infer that \eqref{BandCdec} holds. Conversely, the ``if" part is easy  and is therefore omitted.
\end{proof}
\par
Notice that in the proof of Theorem \ref{rankmMPth1},  the first equation in \eqref{H1H3H2H4} can be replaced   by
\begin{equation}\label{F1F3H1H3re}
\left(\begin{array}{c}
                J_1 \\
           J_3 \\
          \end{array}
        \right)
\left(
  \begin{array}{cc}
    \Sigma K & \Sigma L \\
  \end{array}
\right)
\left(\begin{array}{c}
                H_1 \\
            H_3 \\
          \end{array}
        \right)=\left(
          \begin{array}{c}
                K^*(G_1\Sigma\Delta)^{-1} \\
           L^*(G_1\Sigma\Delta)^{-1} \\
          \end{array}
        \right),
\end{equation}
which is a second order matrix equation. Then, by applying  Lemma \ref{XAYBsolle} to \eqref{F1F3H1H3re}, different characterizations of $B$ and $C$ given by \eqref{BandCdec} are shown in the next theorem.

\begin{theorem}\label{rankchaminMPsecth}
Let $A$ be given in \eqref{HSAdec} with ${\rm rank}(A^{\sim}AA^{\sim})={\rm rank}(A)$, and let $X\in\mathbb{C}^{n\times n}$. Then $X=A^{\mathfrak{m}}$ is the unique solution of  the rank equation  \eqref{rankeq}
if and only if
\begin{equation}\label{Bdeco2}
B=AU\left(
      \begin{array}{cc}
        Q^{-1}\left(
                \begin{array}{c}
                    X_{1}^{-1} \\
                  Y_{3} \\
                \end{array}
              \right)W
         & \left[I_n-\left({\hat C}\left(
                               \begin{array}{cc}
                                 \Sigma L & \Sigma L \\
                               \end{array}
                             \right)\right)^{(1)}{\hat C}\left(
                               \begin{array}{cc}
                                 \Sigma L & \Sigma L \\
                               \end{array}
                             \right)
         \right]Z
          \\
      \end{array}
    \right)
U^*G,
\end{equation}
\begin{equation}\label{Cdeco2}
C=GU{\hat C}\left(\begin{array}{cc}
           \Sigma L & \Sigma L \\
       \end{array}\right)U^*,
\end{equation}
where
 ${\hat C}=S\left(
             \begin{array}{c}
                X_{1} \\
               0\\
             \end{array}
           \right)P^{-1}$,
 $\left({\hat C}\left( \begin{array}{cc}\Sigma L & \Sigma L \\
\end{array} \right)\right)^{(1)}\in \left({\hat C}\left(
 \begin{array}{cc} \Sigma L & \Sigma L \\\end{array}\right)\right)\{1\}$,
$X_{1}\in\mathbb{C}^{r\times r}$ is an arbitrary nonsingular   matrix, $Y_{3}\in\mathbb{C}^{(n-r)\times r}$ and $Z\in\mathbb{C}^{n\times (n-r)}$ are arbitrary,
and  $P,W\in\mathbb{C}^{r\times r}$ and $Q,S\in\mathbb{C}^{n\times n}$ are all nonsingular matrices such that
\begin{equation}\label{PQSTcon}
  P\left(
  \begin{array}{cc}
    I_r & 0  \\
  \end{array}
\right)Q=\left(\begin{array}{cc}
    \Sigma K & \Sigma L  \\
  \end{array}
\right)
,
S\left(
  \begin{array}{c}
 I_r \\
 0  \\
  \end{array}
\right)W=\left(
  \begin{array}{c}
 K^*(G_1\Sigma\Delta){-1} \\
  L^*(G_1\Sigma\Delta){-1}  \\
  \end{array}
\right).
\end{equation}
\end{theorem}

\begin{proof}
For convenience, we use the same  notations as in the proof of Theorem \ref{rankmMPth1}. First, it is clear  to see the existence of  nonsingular matrices $P,W\in\mathbb{C}^{r\times r}$ and $Q,S\in\mathbb{C}^{n\times n}$ satisfying \eqref{PQSTcon}.
To prove the ``only if" part,
applying Lemma \ref{XAYBsolle} to \eqref{F1F3H1H3re}, we have that
\begin{equation}\label{F1F3H1H3sol}
\left(\begin{array}{c}
               J_1 \\
            J_3 \\
          \end{array}
        \right)=S\left(\begin{array}{c}
                X_1 \\
            0 \\
          \end{array}
        \right)P^{-1} ,
\left(\begin{array}{c}
                H_1 \\
            H_3 \\
          \end{array}
        \right)=Q^{-1}\left(\begin{array}{c}
                X_1^{-1} \\
            Y_3 \\
          \end{array}
        \right)W,
\end{equation}
where $X_{1}\in\mathbb{C}^{r\times r}$ is an arbitrary nonsingular   matrix, $Y_{3}\in\mathbb{C}^{(n-r)\times r}$ is arbitrary.
Note that the second equation in \eqref{H1H3H2H4} can be rewritten as
\begin{equation}\label{F1F3H2H4re}
\left(\begin{array}{c}
                J_1 \\
            J_3 \\
          \end{array}
        \right)
\left(
  \begin{array}{cc}
    \Sigma K & \Sigma L \\
  \end{array}
\right)
\left(\begin{array}{c}
                H_2 \\
            H_4 \\
          \end{array}
        \right)=\left(\begin{array}{c}
                0 \\
            0 \\
          \end{array}
        \right).
\end{equation}
Then, substituting first equation in \eqref{F1F3H1H3sol} to \eqref{F1F3H2H4re}, again by Lemma \ref{AXDsolle} we obtain that
\begin{equation}\label{H2H4dec2}
 \left(\begin{array}{c}
                H_2 \\
            H_4 \\
          \end{array}
        \right)=\left[I_n-\left({\hat C}\left(
                               \begin{array}{cc}
                                 \Sigma L & \Sigma L \\
                               \end{array}
                             \right)\right)^{(1)}{\hat C}\left(
                               \begin{array}{cc}
                                 \Sigma L & \Sigma L \\
                               \end{array}
                             \right)
         \right]Z,
\end{equation}
where ${\hat C}=S\left(
             \begin{array}{c}
                X_{1} \\
               0\\
             \end{array}
           \right)P^{-1}$ and
$Z\in\mathbb{C}^{n\times (n-r)}$ is arbitrary. Therefore, applying  \eqref{F1F3H1H3sol} and \eqref{H2H4dec2} to    \eqref{BHSdec} and \eqref{CHSdec},  we infer that \eqref{Bdeco2} and \eqref{Cdeco2} hold. Conversely, the ``if" part is easy.
\end{proof}
\par
It has also drawn a lot of interest to characterize the generalized inverse by using a specific rank equation (see \cite{rankeqchaMPre,rankatsre,solutionDinversere}). At the end of this section, we turn our attention on this consideration.

\begin{theorem}
Let $A\in\mathbb{C}^{m\times n}_r$ with ${\rm rank}({A^{\sim}}AA^{\sim})={\rm rank}(A)$. Then there exist a unique matrix $X\in\mathbb{C}^{n\times n}$ such that
\begin{equation}\label{Xconditioneq}
  AX=0,X^{\sim}=X, X^2=X, {\rm rank}(X)=n-r,
\end{equation}
a unique matrix $Y\in\mathbb{C}^{m\times m}$ such that
\begin{equation}\label{Yconditioneq}
  YA=0,Y^{\sim}=Y, Y^2=Y, {\rm rank}(Y)=m-r,
\end{equation}
and a unique matrix $Z\in\mathbb{C}^{n\times m}$ such that
\begin{equation}\label{rankeqXYZeq}
{\rm rank}\left(
            \begin{array}{cc}
              A & I_m-Y \\
              I_n-X & Z \\
            \end{array}
          \right)={\rm rank}(A).
\end{equation}
Furthermore,  $X=I_n-A^{\mathfrak{m}}A$, $Y=I_m-A A^{\mathfrak{m}}$ and  $Z=A^{\mathfrak{m}}$.
\end{theorem}

\begin{proof}
From $AX=0 $ and $X^{\sim}=X$, we have $\mathcal{R}(X)\subseteq \mathcal{N}(A)$ and $\mathcal{R}(A^{\sim})\subseteq \mathcal{N}(X)$, which, together with ${\rm rank}(X)=n-r$, show $\mathcal{R}(X)= \mathcal{N}(A)$ and $\mathcal{N}(X)=\mathcal{R}(A^{\sim})$. Hence, by $X^2=X$ and the item \eqref{AmAMPeq} in  Lemma \ref{mMPprole}, it follows that the unique solution of \eqref{Xconditioneq} is $X=P_{\mathcal{N}(A),\mathcal{R}(A^{\sim})}=I_n-A^{\mathfrak{m}}A$.
Analogously, we can have that $Y=I_m-A A^{\mathfrak{m}}$ is the unique matrix satisfying \eqref{Yconditioneq}.
Next, it is clear that $\mathcal{R}(I_m-Y) = \mathcal{R}(A)$ and $\mathcal{R}(I_n-X^*)=\mathcal{R}(A^*)$. Thus, applying Lemma \ref{raneconle}, we have that $Z=(I_n-X)A^{\dag}(I_m-Y)= A^{\mathfrak{m}}AA^{\dag}A A^{\mathfrak{m}}= A^{\mathfrak{m}}$ is the unique matrix such that  \eqref{rankeqXYZeq}.
\end{proof}

\section{New  representations of the Minkowski inverse}
\label{generalizingZminMPsec}
Zlobec \cite{ZlobecMPref} established an explicit form of the Moore-Penrose inverse,   also known as Zlobec formula, that is, $A^{\dag}=A^*(A^*AA^*)^{(1)}A^*$,  where $A\in\mathbb{C}^{m\times n}$ and $(A^*AA^*)^{(1)}\in (A^*AA^*)\{1\}$. In this section, we first  present a  more general representation of the Minkowski inverse similar to Zlobec formula.

\begin{theorem}\label{genzlobecmMpth}
Let $A\in\mathbb{C}^{m\times n}$ be such that ${\rm rank}(AA^{\sim})={\rm rank}(A^{\sim}A)={\rm rank}(A)$. Then
\begin{equation*}
A^{\mathfrak{m}}=(A^{\sim}A)^{k}A^{\sim}[(A^{\sim}A)^{k+l+1}A^{\sim}]^{(1)}(A^{\sim}A)^{l}A^{\sim},
\end{equation*}
where $k$ and $l$ are arbitrary nonnegative integers, and $[(A^{\sim}A)^{k+l+1}A^{\sim}]^{(1)}\in [(A^{\sim}A)^{k+l+1}A^{\sim}]\{1\}$.
\end{theorem}

\begin{proof}
First we  use induction on an arbitrary positive integer $s$ to prove that
 ${\rm rank}((A^{\sim}A)^{s})={\rm rank}(A)$. Clearly, ${\rm rank}(A^{\sim}A)={\rm rank}(A)$. Suppose that ${\rm rank}((A^{\sim}A)^{s})={\rm rank}(A)$. Since $\mathcal{R}(A^{\sim})\cap \mathcal{N}(A)=\{0\}$ from ${\rm rank}(AA^{\sim})={\rm rank}(A)$, we infer that
\begin{align*}
{\rm rank}((A^{\sim}A)^{s+1})&={\rm rank}(A^{\sim}A(A^{\sim}A)^{s})={\rm rank}((A^{\sim}A)^{s})-{\rm dim}(\mathcal{R}((A^{\sim}A)^{s})\cap \mathcal{N}(A^{\sim}A))\\
&={\rm rank}(A)-{\rm dim}(\mathcal{R}(A^{\sim})\cap \mathcal{N}(A))={\rm rank}(A),
\end{align*}
which  completes the induction. Hence,  for an arbitrary nonnegative integer $k$, we have that
\begin{equation*}
{\rm rank}(A)={\rm rank}((A^{\sim}A)^{k+1})\leq{\rm rank}((A^{\sim}A)^{k}A^{\sim})\leq {\rm rank}(A),
\end{equation*}
which implies ${\rm rank}((A^{\sim}A)^{k}A^{\sim})= {\rm rank}(A)$. Thus,  ${\rm rank}((A^{\sim}A)^{k+l+1}A^{\sim})={\rm rank}((A^{\sim}A)^{k}A^{\sim})={\rm rank}((A^{\sim}A)^{l}A^{\sim})={\rm rank}(A)$, where $l$ is an  an arbitrary nonnegative integer.
Therefore,  by Lemma \ref{Urquhartth} and \eqref{minMPATS2} in Remark \ref{minMPATS2rem}, it follows that
\begin{align*}
(A^{\sim}A)^{k}A^{\sim}[(A^{\sim}A)^{k+l+1}A^{\sim}]^{(1)}(A^{\sim}A)^{l}A^{\sim}   &=A^{(1,2)}_{\mathcal{R}((A^{\sim}A)^{k}A^{\sim}),\mathcal{N}((A^{\sim}A)^{l}A^{\sim})}\\
   &=A^{(1,2)}_{\mathcal{R}(A^{\sim}),\mathcal{N}(A^{\sim})}\\
   &=A^{\mathfrak{m}},
\end{align*}
where $[(A^{\sim}A)^{k+l+1}A^{\sim}]^{(1)}\in [(A^{\sim}A)^{k+l+1}A^{\sim}]\{1\}$. This now completes the proof.
\end{proof}
\par
Under the hypotheses of Theorem \ref{genzlobecmMpth}, when $k=l=0$, we directly give an explicit expression of the Minkowski inverse in the following corollary.
It is worth mentioning that this result can also be obtained by applying Lemma \ref{Urquhartth} to \eqref{minMPATS2} in Remark \ref{minMPATS2rem}.

\begin{corollary}\label{mMPzlobecremark}
Let $A\in\mathbb{C}^{m\times n}$ be such that ${\rm rank}(AA^{\sim})={\rm rank}(A^{\sim}A)={\rm rank}(A)$. Then
\begin{equation}\label{mMPzlobeceq}
A^{\mathfrak{m}}=A^{\sim}(A^{\sim}AA^{\sim})^{(1)}A^{\sim},
\end{equation}
where $(A^{\sim}AA^{\sim})^{(1)}\in (A^{\sim}AA^{\sim})\{1\}$.
\end{corollary}
\par
Another corollary of Theorem \ref{genzlobecmMpth} given below  shows a different representation of the Minkowski inverse.

\begin{corollary}\label{otherrepminicor}
Let $A\in\mathbb{C}^{m\times n}$ be such that ${\rm rank}(AA^{\sim})={\rm rank}(A^{\sim}A)={\rm rank}(A)$. Then
\begin{equation*}
A^{\mathfrak{m}}=(A^{\sim}A)^{k}A^{\sim}\left[(AA^{\sim})^{k+1}\right]^{(1)} A\left[(A^{\sim}A)^{l+1}\right]^{(1)}(A^{\sim}A)^lA^{\sim},
\end{equation*}
where $k$ and $l$ are arbitrary nonnegative integers,
$\left[(AA^{\sim})^{k+1}\right]^{(1)} \in \left[(AA^{\sim})^{k+1}\right]\{1\}$,  and $\left[(A^{\sim}A)^{l+1}\right]^{(1)}\in \left[(A^{\sim}A)^{l+1}\right]\{1\}$.
\end{corollary}

\begin{proof}
Using Theorem \ref{genzlobecmMpth} and Lemma \ref{mMPprole}, we have that
\begin{align*}
A^{\mathfrak{m}}=&(A^{\sim}A)^{k}A^{\sim}\left[(A^{\sim}A)^{k+l+1}A^{\sim}\right]^{(1)}(A^{\sim}A)^{l}A^{\sim}\\
=&A^{\mathfrak{m}}A(A^{\sim}A)^{k}A^{\sim}\left[(A^{\sim}A)^{k+l+1}A^{\sim}\right]^{(1)}(A^{\sim}A)^{l}A^{\sim}AA^{\mathfrak{m}}\\
=&A^{\mathfrak{m}}A(A^{\sim}A)^{k}A^{\sim}\left[A(A^{\sim}A)^{k}A^{\sim}\right]^{(1)}A(A^{\sim}A)^{k}A^{\sim}\left[(A^{\sim}A)^{k+l+1}A^{\sim}\right]^{(1)}\\ &(A^{\sim}A)^{l}A^{\sim}A\left[(A^{\sim}A)^{l}A^{\sim}A\right]^{(1)}(A^{\sim}A)^{l}A^{\sim}AA^{\mathfrak{m}}\\
=&A^{\mathfrak{m}}A(A^{\sim}A)^{k}A^{\sim}\left[A(A^{\sim}A)^{k}A^{\sim}\right]^{(1)}A\left[(A^{\sim}A)^{l}A^{\sim}A\right]^{(1)}(A^{\sim}A)^{l}A^{\sim}AA^{\mathfrak{m}}\\
=&(A^{\sim}A)^{k}A^{\sim}\left[(AA^{\sim})^{k+1}\right]^{(1)} A\left[(A^{\sim}A)^{l+1}\right]^{(1)}(A^{\sim}A)^lA^{\sim},
\end{align*}
which completes the proof.
\end{proof}

\begin{remark}
Under the hypotheses of Corollary \ref{otherrepminicor}, when $k=l=0$, we have immediately  \cite[Theorem 5]{indef}, that is,
\begin{equation*}
A^{\mathfrak{m}}=A^{\sim}(AA^{\sim})^{(1)}A(A^{\sim}A)^{(1)}A^{\sim},
\end{equation*}
where $(AA^{\sim})^{(1)}\in(AA^{\sim})\{1\}$ and $(A^{\sim}A)^{(1)}\in (A^{\sim}A)\{1\}$.
\end{remark}
\par
This section concludes with  showing  the Minkowski inverse of a class of block matrices  by using  Corollary \ref{mMPzlobecremark}, which extends \cite[Corollary 1]{ZlobecMPref} to the Minkowski inverse.

\begin{theorem}
Let $A\in\mathbb{C}^{m\times n}_r$ be such that ${\rm rank}(AA^{\sim})={\rm rank}(A^{\sim}A)={\rm rank}(A)$ and
\begin{equation}\label{Ablockeq}
A=\left(
    \begin{array}{cc}
      A_{1} & A_{2} \\
      A_{3} & A_{4} \\
    \end{array}
  \right),
\end{equation}
where $A_{1}\in\mathbb{C}^{r\times r}$ is  nonsingular, $A_{2}\in\mathbb{C}^{r\times (n-r)}$, $A_{3}\in\mathbb{C}^{(m-r)\times r}$, and $A_{4}\in\mathbb{C}^{(m-r)\times (n-r)}$. Then
\begin{equation}\label{mMpsepmateq}
A^{\mathfrak{m}}=
\left(
  \begin{array}{cc}
   A_{1} & A_{2} \\
  \end{array}
\right)^{\sim}
\left[
\left(
  \begin{array}{c}
   A_{1} \\
    A_{3} \\
  \end{array}
\right)^{\sim}
A
\left(
  \begin{array}{cc}
   A_{1} & A_{2} \\
  \end{array}
\right)^{\sim}
\right]^{-1}
\left(
  \begin{array}{c}
   A_{1} \\
    A_{3} \\
  \end{array}
\right)^{\sim}.
\end{equation}
\end{theorem}

\begin{proof}
Let
\begin{equation}\label{TA1A2AA1A2eq}
  T_{1}=\left(
  \begin{array}{cc}
   A_{1} & A_{2} \\
  \end{array}
\right)
A^{\sim}
\left(
  \begin{array}{c}
   A_{1} \\
    A_{3} \\
  \end{array}
\right).
\end{equation}
Since $A_{1}\in\mathbb{C}^{r\times r}$ is  nonsingular, we have
\begin{align}
{\rm rank}(A)&=
{\rm rank}\left(\left(
    \begin{array}{cc}
      I_{r}& 0 \\
     -A_{3}A_{1}^{-1} & I_{m-r} \\
    \end{array}
  \right)\left(
    \begin{array}{cc}
      A_{1} & A_{2} \\
      A_{3} & A_{4} \\
    \end{array}
  \right)\left(
    \begin{array}{cc}
      I_{r}& -A_{1}^{-1}A_{2} \\
     0 & I_{n-r} \\
    \end{array}
  \right)
  \right)\nonumber\\
  &={\rm rank}
  \left(
    \begin{array}{cc}
      A_{1} & 0 \\
     0 & A_{4}-A_{3}A_{1}^{-1}A_{2} \\
    \end{array}
  \right),\label{A22A21A11A12eq}
\end{align}
which, together with ${\rm rank}(A)={\rm rank}(A_{1})$, gives $A_{4}=A_{3}A_{1}^{-1}A_{2}$. Then, it can  be easily verified that
\begin{equation}\label{T11redeceq}
T_{1}=\left(
  \begin{array}{cc}
   A_{1} & A_{2} \\
  \end{array}
\right)
\left(
  \begin{array}{cc}
   A_{1} & A_{2} \\
  \end{array}
\right)^{\sim}{(A_{1}^{\sim})}^{-1}
\left(
  \begin{array}{c}
   A_{1} \\
    A_{3} \\
  \end{array}
\right)^{\sim}
\left(
  \begin{array}{c}
   A_{1} \\
    A_{3} \\
  \end{array}
\right).
\end{equation}
 It is sufficient to prove that $T_1$ is nonsingular, i.e.,  ${\rm rank}(T_{1})=r$. In fact, since $\mathcal{N}(A)=\mathcal{N}\left(\left(
  \begin{array}{cc}
   A_{1} & A_{2} \\
  \end{array}
\right)\right)$ from the nonsingularity of $A_{1}$, we have $\mathcal{R}(A^{\sim})=\mathcal{R}\left(\left(
  \begin{array}{cc}
   A_{1} & A_{2} \\
  \end{array}\right)^{\sim}\right)$.
  Then, since $\mathcal{R}(A^{\sim})\cap\mathcal{N}(A)=\{0\}$ from ${\rm rank}(AA^{\sim})={\rm rank}(A)$, we infer that
\begin{align}
  &{\rm rank}\left(\left(
  \begin{array}{cc}
   A_{1} & A_{2} \\
  \end{array}
\right)\left(
  \begin{array}{cc}
   A_{1} & A_{2} \\
  \end{array}
\right)^{\sim}\right)\nonumber \\
=&{\rm rank}\left(\left(
  \begin{array}{cc}
   A_{1} & A_{2} \\
  \end{array}
\right)\right)-{\rm dim}\left(\mathcal{R}\left(\left(
  \begin{array}{cc}
   A_{1} & A_{2} \\
  \end{array}\right)^{\sim}\right)\cap\mathcal{N}\left(\left(
  \begin{array}{cc}
   A_{1} & A_{2} \\
  \end{array}
\right)\right)\right)  \nonumber\\
=& {\rm rank}(A_{1})-{\rm dim}\left(\mathcal{R}(A^{\sim})\cap\mathcal{N}(A)\right)=r.\label{rankA11A12}
\end{align}
Analogously, we can obtain that
\begin{equation}\label{rankA11A21}
{\rm rank}\left(\left(
  \begin{array}{c}
   A_{1} \\
    A_{3} \\
  \end{array}
\right)^{\sim}
\left(
  \begin{array}{c}
   A_{1} \\
    A_{3} \\
  \end{array}
\right)\right)=r.
\end{equation}
In terms of \eqref{T11redeceq}, \eqref{rankA11A12} and \eqref{rankA11A21}, it is clear that ${\rm rank}(T_{1})=r$. Then using the item \eqref{minMPexistATSitem3} in Theorem  \ref{mMPexitscondition}, we get ${\rm rank}(T_{1})={\rm rank}(A^{\sim}AA^{\sim})$.
Denote
  \begin{equation*}
\left(
  \begin{array}{cc}
   B_{1} & B_{2} \\
    B_{3}  & B_{4}  \\
  \end{array}
\right)=A^{\sim}AA^{\sim},
  \end{equation*}
where $B_{1}\in\mathbb{C}^{r\times r}$, $B_{2}\in\mathbb{C}^{r\times (m-r)}$, $B_{3}\in\mathbb{C}^{(n-r)\times r}$ and $B_{4}\in\mathbb{C}^{(n-r)\times (m-r)}$.
In view of  \eqref{TA1A2AA1A2eq}, we see that $B_{1}=T_{1}^{\sim}$. Thus, by the same way of \eqref{A22A21A11A12eq}, we have   that $B_{4}=B_{3}(T_{1}^{\sim})^{-1}B_{2}$. Then it is easy to prove that
  \begin{equation}\label{T1AAAAeq}
\left(
  \begin{array}{cc}
    (T_{1}^{\sim})^{-1} & 0 \\
    0 & 0 \\
  \end{array}
\right)\in \left(A^{\sim}AA^{\sim}\right)\{1\}.
  \end{equation}
Substituting \eqref{T1AAAAeq} and \eqref{Ablockeq} to \eqref{mMPzlobeceq} in Corollary \ref{mMPzlobecremark},  we have \eqref{mMpsepmateq} by direct calculation. This completes the proof.
\end{proof}

\section{Conclusion}\label{conlusionsec}
This paper shows some different characterizations and representations of the Minkowski inverse in Minkowski space, mainly by extending some known results of the Moore-Penrose inverse to the Minkowski inverse.
In addition,  we are convinced that  the study of generalized inverses in Minkowski space will maintain its popularity for years to come.
Several possible directions  for further research can be described as follows:
\begin{enumerate}[$(1)$]
  \item It is difficult but interesting to explore the representation of the Minkowski inverse by using core-EP decomposition \cite{coreepdecre}.
  \item  One possibility is to establish more characterizations and representations of the Minkowski inverse in terms of  the results of
      the Moore-Penrose inverse in other mathematical fields introduced in Section \ref{introductionsection}.
  \item As we know,  the study of the Minkowski inverse originates from the simplification of polarized light  problems \cite{minSVDre}. A   meaningful research topic   is to find out new applications of the Minkowski inverse in the study on   polarization of light by using its existing mathematical  results.
\end{enumerate}

\end{document}